\documentclass{amsart}
\usepackage{amsmath}%
\usepackage{amsfonts}%
\usepackage{amssymb}% 
\usepackage{amsthm}
\usepackage{amscd}
\usepackage{graphicx}
\usepackage{mathrsfs}
\usepackage{xcolor}
\usepackage{makecell}
\usepackage{tikz}
\usetikzlibrary{positioning}
\usepackage{comment}
\usepackage{url}
\usepackage[T1]{fontenc}
\usepackage{lmodern}
\usepackage{float}
\usepackage{newtxtext}
\usepackage[colorlinks=true,
            linkcolor=blue,
            citecolor=blue,
            urlcolor=blue]{hyperref}
\usepackage{subcaption}

\usepackage{soul}
\newtheorem{theorem}{Theorem}[section]

\newtheorem{corollary}[theorem]{Corollary}
\newtheorem{lemma}[theorem]{Lemma}
\newtheorem{proposition}[theorem]{Proposition}
\newtheorem{definition}[theorem]{Definition}
\newtheorem{example}[theorem]{Example}
\newtheorem{remark}[theorem]{Remark}

\newtheorem*{definition*}{Definition}
\newtheorem*{proposition*}{Proposition}
\newtheorem*{theorem*}{Theorem}
\newtheorem*{lemma*}{Lemma}
\newtheorem*{corollary*}{Corollary}
\newtheorem*{remark*}{Remark}

\numberwithin{equation}{section}

\usepackage{scalerel,stackengine}
\stackMath
\newcommand\reallywidehat[1]{%
\savestack{\tmpbox}{\stretchto{%
  \scaleto{%
    \scalerel*[\widthof{\ensuremath{#1}}]{\kern.1pt\mathchar"0362\kern.1pt}%
    {\rule{0ex}{\textheight}}%WIDTH-LIMITED CIRCUMFLEX
  }{\textheight}% 
}{2.4ex}}%
\stackon[-6.9pt]{#1}{\tmpbox}%
}
\begin{document}
\title{A Combinatorial Classification of Graded n-Absorbing Principal Monomial Ideals}

\author{Alison Becker}
\address{Independent Researcher, Wisconsin, USA}
\email{alison.elaine90@gmail.com}

\author{Thomas Stojsavljevic}
\address{Department of Math and Computer Science, Beloit College}
\email{stojsavljevictg@beloit.edu}


\begin{abstract}
The theory of $n$-absorbing ideals provides a generalization of prime ideals and has motivated several investigations into their structural properties. In this paper, we study graded $n$-absorbing behavior for principal monomial ideals in polynomial rings and develop a combinatorial model for their classification. Let $R=F[x_{1},x_{2},\dots,x_{m}]$ be a polynomial ring over a field with the standard grading. We prove that a principal monomial ideal $I=(x^{\alpha})$ is graded $n$-absorbing if and only if the total degree of its generating monomial satisfies $\vert \alpha \vert \leq n$. Combining this characterization with results on $n$-absorbing and strongly $n$-absorbing ideals, we show that these three notions coincide for principal monomial ideals. This characterization yields a geometric correspondence between principal monomial ideals and lattice points in $\mathbb{N}^{m}$, under which graded $n$-absorbing ideals correspond precisely to the lattice points contained in the discrete simplex $\Delta_n^m=\{\alpha\in\mathbb{N}^{m}:|\alpha|\leq n\}.$ We further identify the associated divisibility poset with the Cayley graph of $(\mathbb{N}^{m},+) $and derive consequences for polynomial extensions, enumeration, and homological invariants.                                                                                         
\end{abstract}

\maketitle

\section{Introduction}

The theory of prime ideals is fundamental to commutative algebra, providing a framework for understanding the multiplicative structure of rings and their associated geometric properties. Recall, for a proper ideal $P$ of commutative ring $R$, $P$ is a prime ideal if whenever $ab\in P$, either $a\in P$ or $b\in P$. In this paper, we study an extension of prime ideals called $n$-absorbing ideals, and provide an explicit combinatorial classification of their characteristics in the case of principal monomial ideas. A proper ideal $I$ is $n$-absorbing if, whenever a product of $n+1$ elements belongs to $I$, some product obtained by omitting one factor also belongs to $I$ \cite{anderson_n_2011}. The case $n=1$ recovers prime ideals, while larger values of $n$ provide a hierarchy of increasingly weaker absorption conditions.

Since their introduction, $n$-absorbing ideals have been investigated from several perspectives, including $n$-absorbing primary ideals, graded 2-absorbing ideals, weakly absorbing ideals, strongly absorbing ideals, and factorization-theoretic generalizations \cite{badawi_2-absorbing_2014, becker_results_2015, naghani2016graded, al2019graded, al2017graded, akray2024graded, yousefian2024graded, choi2024n, secord24}. Despite this activity, explicit classifications of $n$-absorbing ideals remain relatively uncommon, particularly in graded settings where additional algebraic structure suggests the possibility of concrete combinatorial descriptions. This observation motivates the present work.

Our contributions are twofold. First, we develop a structural theory of graded $n$-absorbing ideals in $G$-graded commutative rings. Extending the framework previously established for graded $2$-absorbing ideals, we investigate graded $n$-absorbing, graded weakly $n$-absorbing, and graded strongly $n$-absorbing ideals together with their behavior under graded homomorphisms, quotient constructions, and homogeneous decompositions. We revisit the notion of a graded $2$-absorbing subgroup introduced by Naghani and Moghimi \cite{naghani2016graded}, identify an issue with its formulation, and provide a corrected definition that extends naturally to arbitrary $n$.The issue with the definition in \cite{naghani2016graded} motivates the following corrected formulation.
\begin{definition*}
Let $R=\bigoplus_{g\in G}R_{g}$ be a graded ring and let $I = \bigoplus_{g \in G} I_{g}$ be graded ideal of $R$. For $g \in G$, we say that the homogeneous component $I_{g}$ is an \textbf{\textit{n}-absorbing homogeneous component} of $I$ if $I_{g} \neq R_{g}$ and whenever $x_{1},x_{2},\ldots,x_{n+1}\in h(R)$ for $x_{1}x_{2}\cdots x_{n+1} \in I_{g}$, then $x_1x_2\cdots\reallywidehat{x_i}\cdots x_{n+1}\in I$ for some $i \in \{1, 2, \dots, n+1\}$.
\end{definition*}
\noindent The following observation shows that this definition provides the desired componentwise formulation of graded $n$-absorption.
\begin{proposition*}
Let $R=\bigoplus_{g\in G}R_g$ be a $G$-graded commutative ring and let $I=\bigoplus_{g\in G}I_g$ be a proper graded ideal of $R$. Then $I$ is graded $n$-absorbing if and only if every homogeneous component $I_g$ is an $n$-absorbing homogeneous component of $I$.
\end{proposition*}
\noindent Thus, the graded $n$-absorbing property of a graded ideal can be characterized componentwise, allowing the behavior of such ideals to be studied through their homogeneous components.

We continue to explore $n$-absorbing characteristics of graded rings through invariance and monotonicity results, which builds a foundation upon which we illustrate that the following open conjecture of \cite{anderson_n_2011} holds in the setting of principal monomial ideals: 
\begin{theorem*}
    Let $F$ be a field and $R = F[x_1, \dots, x_m]$. For every proper, principal monomial ideal $I = (x^{\alpha})$ of $R$, $\omega_{R[X]}(I[X]) = \omega_{R}(I).$
\end{theorem*}
\noindent Motivated by our study of principal monomial ideals, we establish a complete combinatorial characterization of graded $n$-absorbing principal monomials in standard-graded polynomial rings. Identifying a principal monomial $x^{\alpha}$ with its exponent vector $\alpha\in\mathbb{N}^{m}$, gives a complete geometric characterization of the graded $n$-absorbing property by proving that $x^{\alpha}$ is graded $n$-absorbing if and only if $\alpha \in \Delta_{n}^{m}$ 
where
\[
\Delta_n^m=\{\alpha\in\mathbb{N}^{m}: \vert \alpha \vert \leq n\}.
\]
Thus, the graded $n$-absorbing property is completely determined by membership in a finite simplex of lattice points. This characterization transforms an algebraic absorption condition into a geometric constraint and provides a natural combinatorial realization of graded $n$-absorbing principal monomial ideals. 

We classify the graded $n$-absorbing properties of principal monomial ideals via a combinatorial dissection of lattice points in $\mathbb{N}^m$, 
\begin{theorem*}
    Let $F$ be a field and let $R=F[x_{1},x_{2},\dots,x_{m}]$ with the standard grading. For $I=(x^{\alpha})$, the following are equivalent:
    \begin{enumerate}
        \item $I$ is graded $n$-absorbing.
        \item The exponent vector satisfies $\vert \alpha \vert \leq n$.
        \item The vector $\alpha$ lies in the discrete simplex
        \[
        \Delta_{n}^{m} = \{\alpha \in \mathbb{N}^{m}: \vert \alpha \vert \leq n\}.
        \]
    \end{enumerate}
\end{theorem*}
\noindent We further interpret the poset of principal monomial ideals through the Cayley graph of the additive monoid $(\mathbb N^m,+)$, identifying the graded $n$-absorbing ideals as an induced finite subgraph. This perspective establishes new connections between graded $n$-absorbing ideals, combinatorial commutative algebra, and classical graded invariants including Hilbert series, Krull dimension, and graded Betti numbers.

The paper is organized as follows. Section 2 establishes notation and recalls the necessary background on graded commutative rings, graded $n$-absorbing ideals, and principal monomial ideals. Section 3 develops the structural theory of graded $n$-absorbing ideals in $G$-graded commutative rings. Section 4 develops the combinatorial model for principal monomial ideals, proves the simplex characterization and Anderson-Badawi Conjecture III for this class, and investigates the resulting connections with Cayley graphs and classical graded invariants. We conclude with several illustrative examples and directions for future research.

Throughout this paper all rings are assumed to be commutative with identity. We begin by recalling the standard definitions concerning graded commutative rings and graded ideals that will be used throughout.

\section{Background and Notation}

Let $G$ be a group with identity $e$, we say that a $G$-graded ring is a ring $R$ together with a decomposition
\begin{equation}
    R = \bigoplus_{g \in G} R_{g},
\end{equation}
as a $\mathbb{Z}$-module such that $R_{g}R_{h} \subseteq R_{g+h}$ for all $g,h \in G$ . We denote the graded ring $R$ together with its $G$-grading by $G(R)$. The summands $R_{g}$ are called the homogeneous components and the elements of the summands are called homogeneous elements of degree $g$. The set of homogeneous elements of $R$, denoted by $h(R)$, is given by
\begin{equation}
    h(R) = \bigcup_{g \in G} R_{g}.
\end{equation}
From this construction, if $a \in G(R)$, then $a$ can be written uniquely as 
\begin{equation}
    a = \sum_{g \in G} a_{g},
\end{equation}
where $a_{g}$ is called the $g$-th homogeneous component of $a$. Additionally, if $R$ is a G-graded ring, then $R_{e}$ is a subring of $R$ and $R_{g}$ is an $R_{e}$-module for all $g \in G$ \cite{eisenbud2013commutative}. 

An ideal $I$ of a graded ring $R$ is said to be a graded ideal (or homogeneous ideal) whenever
\begin{equation}
    I = \bigoplus_{g \in G} \left(I \cap R_{g}\right) =\bigoplus_{g \in G} I_{g}.
\end{equation}
Equivalently, we say $I$ is a graded ideal of $R$ if every element $a \in I$, all the homogeneous elements of $a$ are in $I$. Finally, a graded ideal $P$ of a G-graded ring $G(R)$ is said to be a graded prime ideal if $P \neq R$ and whenever $ab \in P$ for $a,b \in h(R)$ then either $a \in P$ or $b \in P$ . Similarly, we can define a weakly graded prime ideal by adding the condition that $0 \neq ab \in P$ to the previous definition \cite{eisenbud2013commutative, northcott_homogeneous_1955}. 

Throughout the remainder of the paper,
\begin{itemize}
\item $R=\bigoplus_{g\in G}R_g$ denotes a $G$-graded commutative ring;

\item $h(R)$ denotes the set of homogeneous elements;

\item all graded ideals are assumed proper unless stated otherwise.

\end{itemize}

\subsection{\textit{n}-Absorbing Ideals and \textit{n}-Absorbing Primary Ideals}

In \cite{anderson_n_2011}, Anderson and Badawi defined an $n$-absorbing ideal of a commutative ring $R$ in the following fashion.

\begin{definition}
A proper ideal $I$ of a commutative ring $R$ is called an
\textbf{\textit{n}-absorbing ideal} if whenever
\[
a_{1}a_{2}\cdots a_{n+1}\in I
\]
for $a_{1},\dots,a_{n+1}\in R$, then
\[
a_{1}\cdots\reallywidehat{a_i}\cdots a_{n+1}\in I
\]
for some $i\in\{1,\dots,n+1\}$.
\end{definition}

In \cite{becker_results_2015}, Becker defined a natural extension of 2-absorbing primary ideals studied by Badawi et al in \cite{badawi_2-absorbing_2014} into $n$-absorbing primary ideals. 

\begin{definition}
Let $n$ be a positive integer. A proper ideal $I$ of a commutative ring $R$ is an \textbf{\textit{n}-absorbing primary ideal of \textit{R}} if, whenever $x_{1},x_{2}, \dots, x_{n+1} \in R$ and $x_{1}x_{2}\dots x_{n+1} \in I$, then either $x_{1}x_{2}\dots x_{n} \in I$ or a product of $n$ of the $x_{i}$'s (other than $x_{1}x_{2}\dots x_{n}$) is in $\sqrt{I}$.
\label{def:nabsprim1}
\end{definition}

\noindent Equivalently, $n$-absorbing primary ideals can be defined in the following fashion.

\begin{definition}
A proper nonzero ideal $I$ of $R$ is called an \textbf{\textit{n}-absorbing primary ideal}
if whenever $x_1,x_2,\dots,x_{n+1}\in R$ and
\[
x_1x_2\cdots x_{n+1} \in I,
\]
then either
\[
x_1\cdots \widehat{x_k}\cdots x_{n+1} \in I
\]
for some $k\in[1,n+1]$, or
\[
x_1\cdots \reallywidehat{x_i}\cdots x_{n+1} \in \sqrt{I}
\quad \text{and} \quad
x_1\cdots \reallywidehat{x_j}\cdots x_{n+1} \in \sqrt{I}
\]
for some $i,j\in[1,n+1]$ with $i\neq j$.
\label{def:nabsprim2}
\end{definition}

We see that if there are no products of $n$ elements in $I$, then we have $n$ products of the form $x_{1}x_{2}\dots x_{n+1} \in I$. That is, $x_{2}x_{3}\dots x_{n+1}x_{1} \in I$, $x_{3}x_{4}\dots x_{n+1}x_{1}x_{2} \in I$, and so forth. Thus, we see that if any two products of $n$ elements are in $\sqrt{I}$, there is a product of $n$ elements in $\sqrt{I}$ for each case.

\subsection{Monomial Ideals}

Consider the polynomial ring $F[x_{1},\dots,x_{m}]$ in $m$ variables over a field $F$, equipped with the standard grading
\[
\deg(x_i)=1
\]
for each $i=1,\dots,m$. For a vector 
\[
\alpha=(a_{1},\ldots,a_{m})\in\mathbb{N}^{m},
\]
we write
\[
x^{\alpha}=x_{1}^{a_1}\cdots x_{m}^{a_m}.
\]
The vector $\alpha$ is called the exponent vector of the monomial $x^{\alpha}$, and its total degree is given by
\[
|\alpha|=a_{1}+\cdots+a_{m}.
\]
For exponent vectors $\alpha,\beta\in\mathbb{N}^{m}$, divisibility of monomials is determined by the coordinatewise partial order:
\[
x^\alpha \mid x^\beta \quad \text{if and only if} \quad \alpha_{i}\leq \beta_{i}
\]
for all $i=1,\dots,m$.

\begin{definition}
Let $(P,\leq)$ be a partially ordered set. The \textbf{Hasse diagram} of $P$ is the graph whose vertices are the elements of $P$, with an edge between $x$ and $y$ whenever one covers the other; that is, whenever $x<y$ and there is no $z\in P$ such that $x<y<z$. When the partial order is represented geometrically, the larger element is conventionally placed above the smaller element.
\end{definition}

For the divisibility poset of principal monomial ideals, the cover relations
correspond to multiplication by a single variable.

\subsection{Cayley Graphs}
The notion of a Cayley graph originates in group theory (see, for example,\cite{godsil2013algebraic}). Since our setting is the additive monoid $\mathbb{N}^m$, we adopt the following natural extension of this construction.

\begin{definition}
Let $M$ be an additive monoid generated by a subset $S\subseteq M$. The (directed) Cayley graph of $M$ with respect to $S$, denoted $\operatorname{Cay}(M,S)$, is the directed graph whose vertex set is $M$, with a directed edge
\[
\alpha \longrightarrow \alpha+s
\]
for every $\alpha\in M$ and $s\in S$.
\end{definition}
Throughout this paper we regard $\operatorname{Cay}(M,S)$ as its underlying undirected graph, since edge orientation plays no role in our arguments. In particular, we consider the Cayley graph
\[
\operatorname{Cay}(\mathbb N^m,\{e_1,\ldots,e_m\}),
\]
where $e_1,\ldots,e_m$ are the standard basis vectors of $\mathbb N^m$.

\subsection{Krull Dimension and Graded Resolutions}

We use the following standard definition of Krull dimension \cite{eisenbud2013commutative}.

\begin{definition*}
The \textbf{Krull dimension} of a commutative ring $R$, denoted $\dim(R)$, is the supremum of the lengths of all chains of prime ideals
\[
\mathfrak{p}_0\subsetneq\mathfrak{p}_1\subsetneq\cdots\subsetneq\mathfrak{p}_n
\]
in $R$. Thus, a chain with $n+1$ prime ideals has length $n$.
\end{definition*}

We next recall the definitions of graded free resolutions and minimal graded free resolutions \cite{peeva2010graded}.

\begin{definition*}
Let $S=F[x_1,\ldots,x_m]$ be a standard graded polynomial ring and let $M$ be a finitely generated graded $S$-module. A \textbf{graded free resolution} of $M$ is an exact sequence
\[
\cdots \longrightarrow F_2 \longrightarrow F_1
\longrightarrow F_0 \longrightarrow M \longrightarrow 0,
\]
where each $F_i$ is a graded free $S$-module and each homomorphism in the resolution is homogeneous of degree $0$.
\end{definition*}

\begin{definition*}
A graded free resolution
\[
\mathcal{F}:\quad \cdots \longrightarrow F_{2}
\longrightarrow F_{1} \longrightarrow F_{0} \longrightarrow M
\longrightarrow 0
\]
of a finitely generated graded $S$-module $M$ is called a \textbf{minimal graded free resolution} if, for every $i\geq 1$,
\[
\operatorname{im}(d_i)\subseteq \mathfrak{m}F_{i-1},
\]
where $d_i:F_i\to F_{i-1}$ is the differential and $\mathfrak{m}=(x_1,\ldots,x_m)$ is the homogeneous maximal ideal of $S$.
\end{definition*}

For $j\in\mathbb{Z}$, the graded shift $S(-j)$ denotes the graded $S$-module whose degree-$d$ component is $S_{d-j}$. Equivalently, the free generator of $S(-j)$ has degree $j$.

\begin{definition*}
Let $M$ be a finitely generated graded $S$-module, and let
\[
\mathcal{F}:\quad \cdots \longrightarrow F_{2}
\longrightarrow F_{1} \longrightarrow F_{0} \longrightarrow M
\longrightarrow 0
\]
be its minimal graded free resolution. Each free module $F_i$ can be written in the form
\[
F_{i}=\bigoplus_{j\in\mathbb{Z}} S(-j)^{\beta_{i,j}(M)}.
\]
The nonnegative integers $\beta_{i,j}(M)$ are called the \textbf{graded Betti numbers} of $M$.
\end{definition*}

\section{Results on Graded \textit{n}-Absorbing Ideals}

The notion of a graded $2$-absorbing ideal was introduced in
\cite{al2019graded}. In this section, we extend this notion to arbitrary positive integers $n$ and investigate the resulting graded $n$-absorbing, graded weakly $n$-absorbing, and graded strongly $n$-absorbing ideals.

\begin{definition}
A graded ideal $I$ of $G(R)$ is said to be a \textbf{graded \textit{n}-absorbing ideal} if $I \neq R$ and whenever 
\[
a_{1}a_{2}\dots a_{n+1} \in I
\]
for $a_{1}, a_{2}, \dots, a_{n+1} \in h(R)$, then 
\[
a_{1}a_{2}\dots\reallywidehat{a_{i}}\dots a_{n}a_{n+1} \in I
\]
for some $i \in \{1,2,\dots,n+1\}$.
\end{definition}

\begin{definition}
 A graded ideal $I$ of $G(R)$ is said to be a \textbf{graded weakly n-absorbing ideal} if $I\neq R$ and whenever 
 \[
 a_1a_2\cdots a_{n+1} \in I
 \]
for $a_1, a_2, \dots, a_{n+1} \in h(R)$ with 
 \[
a_1a_2 \cdots a_{n+1} \neq 0
 \]
 then
 \[
a_1a_2\cdots \reallywidehat{a_i}\cdots a_{n+1} \in I
 \]
for some $i \in \{1,2, \dots, n+1\}$.
\end{definition}

\begin{definition}
A proper, graded ideal $I$ of $R$ is \textbf{graded strongly \textit{n}-absorbing} if whenever 
\[
I_1\cdots I_{n+1} \subseteq I
\]
for graded ideals $I_1, \dots, I_{n+1}$ of $R$, then 
\[
I_1\cdots \reallywidehat{I_k}\cdots I_{n+1} \subseteq I .
\]
for some $k \in \{1, 2, \dots, n+1\}$.
\end{definition}

\begin{lemma}\label{lem:ntogradedn}
Let $R=\bigoplus_{g\in G}R_g$ be a graded commutative ring, and let $I$ be a graded ideal of $R$. If $I$ is an $n$-absorbing ideal of $R$, then $I$ is a graded $n$-absorbing ideal of $R$.
\end{lemma}

\begin{proof}
Suppose $I$ is an $n$-absorbing ideal of $R$. Let
\[
a_1,a_2,\dots,a_{n+1}\in h(R)
\]
be homogeneous elements such that
\[
a_1a_2\cdots a_{n+1}\in I.
\]
Since $I$ is $n$-absorbing, there exists some $i\in\{1,\dots,n+1\}$ such that
\[
a_1\cdots\widehat{a_i}\cdots a_{n+1}\in I.
\]
Because the $a_i$ are homogeneous, this is exactly the condition required for $I$ to be graded $n$-absorbing. Therefore, $I$ is a graded $n$-absorbing ideal.
\end{proof}

\begin{remark}
These two properties of ideals, $n$-absorbing and graded $n$-absorbing, should be distinguished carefully. A graded $n$-absorbing ideal is a graded ideal in which we test the $n$-absorbing property only on homogeneous elements, whereas an $n$-absorbing ideal is defined using arbitrary ring elements. Thus, these two notions should not be identified without additional hypotheses. In particular, the graded hypothesis on the ideal is essential: if $I$ is a graded ideal and is an $n$-absorbing ideal of $R$, then $I$ is a graded $n$-absorbing ideal. However, the converse need not hold. An example of a graded $2$-absorbing ideal that is not a $2$-absorbing ideal is given in \cite{al2019graded}.
\end{remark}

\begin{example}[An $n$-absorbing ideal need not be graded]
Consider the ideal
\[
I=(x^2+y^2+xy+1)
\]
in $\mathbb Z[x,y]$. The generator is not homogeneous, so $I$ is not a graded ideal with respect to the standard grading. Since $x^2+y^2+xy+1$ is irreducible and $\mathbb Z[x,y]$ is a UFD, $I$ is prime. Hence $I$ is an $n$-absorbing ideal for every positive integer $n$, but it is not a graded $n$-absorbing ideal.
\end{example}

Structural properties of graded weakly 2-absorbing ideals have been studied in several algebraic settings, including primary ideals \cite{badawi_2-absorbing_2014, wasadikar_properties_2019}, commutative semigroups \cite{yousefian_darani_2-absorbing_2013}, and graded rings \cite{naghani2016graded, al2019graded, al2017graded}. Motivated by these developments, we investigate whether the relationship between graded weakly $n$-absorbing and graded $n$-absorbing ideals exhibits the same behavior for arbitrary absorbing index $n$. This distinction is essential for the remainder of our paper, since the combinatorial constructions developed later occur in reduced graded polynomial rings, where the two notions coincide. 

\subsection{Nilpotent Homogeneous Components}
Existing results for graded weakly 2-absorbing ideals show that the difference between graded weakly 2-absorbing and graded 2-absorbing ideals is controlled by nilpotent homogeneous components. We generalize this phenomenon by showing that the same dichotomy persists for arbitrary absorbing index $n$.

\begin{theorem}\label{Thm:PrimaryResult1}
Let $\displaystyle I = \bigoplus_{g \in G} I_{g}$ be a graded weakly $n$-absorbing ideal of $G(R)$. Then exactly one of the following holds:
\begin{enumerate}
    \item $I$ is a graded $n$-absorbing ideal.
    \item $I_{g}^{n+1} = 0$ for every $g \in G$.
\end{enumerate}
\end{theorem}

The full proof proceeds by induction on $n$. Since each inductive step repeats the same argument with additional notation, we present only the essential mechanism and omit the routine bookkeeping.

\begin{proof}[Proof Sketch]
 Let $I_{g}^{n+1} \neq (0)$ for some $g \in G$. Let $x_{1}, x_{2}, \dots, x_{n+1} \in h(R)$ be such that $x_{1}x_{2}\dots x_{n+1} \in I$. If $x_{1}x_{2}\dots x_{n+1} \neq 0$, and since $I$ is graded weakly $n$-absorbing, we have $x_{1}x_{2}\dots \reallywidehat{x_{i}}\dots x_{n}x_{n+1} \in I$. Thus $I$ is a graded $n$-absorbing ideal.

Now consider the case when $x_{1}x_{2}\dots x_{n+1} = 0$. Without loss of generality, let $x_{n+1}$ be excluded. If $(x_{1}x_{2}\dots x_{n})I_{g} \neq (0)$ then there exists an element $a_{1} \in I_{g}$ such that $x_{1}x_{2}\dots x_{n}a_{1} \neq 0$. Since
\begin{align*}
    0 \neq x_{1}x_{2}\dots x_{n}a_{1} & = x_{1}x_{2}\dots x_{n}a_{1} + x_{1}x_{2}\dots x_{n}x_{n+1} \\
    & = (x_{1}x_{2}\dots x_{n})(a_{1}+x_{n+1}),
\end{align*}
we must have $(x_{1}x_{2}\dots x_{n})(a_{1}+x_{n+1}) \in I$. Thus, we must have that either 
\begin{enumerate}
    \item $x_{1}x_{2}\dots x_{n} \in I$, or
    \item $(x_{1}x_{2}\dots \reallywidehat{x_{i}} \dots x_{n-1}x_{n})(a_{1}+x_{n+1}) \in I$, for some $i = 1, 2, \dots, n$.
\end{enumerate}
If the first case holds, then $I$ is a graded $n$-absorbing as desired. Otherwise, if 
\begin{equation*}
    (x_{1}x_{2}\dots \reallywidehat{x_{i}} \dots x_{n-1}x_{n})(a_{1}+x_{n+1}) \in I,
\end{equation*}
then it follows that
\begin{align*}
    (x_{1}x_{2}\dots \reallywidehat{x_{i}} \dots x_{n-1}x_{n})(a_{1}+x_{n+1}) = \ &(x_{1}x_{2}\dots \reallywidehat{x_{i}} \dots x_{n-1}x_{n})a_{1} \\
    & + (x_{1}x_{2}\dots \reallywidehat{x_{i}} \dots x_{n-1}x_{n})x_{n+1}. 
\end{align*}
Thus $x_{1}x_{2} \dots \reallywidehat{x_{i}} \dots x_{n}x_{n+1} \in I$ as desired, and so $I$ is $n$-absorbing graded ideal. Now we may assume that $(x_{1}x_{2} \dots \reallywidehat{x_{i}} \dots x_{n}x_{n+1})I_{g} = (0).$

Supposing that $(x_{1}x_{2} \dots \reallywidehat{x_{1}} \dots x_{n}x_{n+1})I_{g} = (0)$ and, without loss of generality, that $(x_{1}x_{2}\dots x_{n-2}x_{n-1})I_{g}^{2} \neq 0$, then there exist elements $a_{1},a_{2} \in I_{g}$ such that $(x_{1}\dots x_{n-1})a_{1}a_{2} \neq 0.$

Arguing as before, we obtain
\begin{align*}
    0 & \neq (x_{1}x_{2}\dots x_{n-2}x_{n-1})a_{1}a_{2} \\
    & = x_{1}x_{2}\dots x_{n-2}x_{n-1}a_{1}a_{2} + (x_{1}x_{2}\dots x_{n}\reallywidehat{x_{n+1}}) + (x_{1}x_{2}\dots \reallywidehat{x_{n}}x_{n+1}) \\
    & = (x_{1}x_{2}\dots x_{n-2}x_{n-1})(a_{1}+x_{n})(a_{2}+x_{n+1)}.
\end{align*}
Hence, it is the case that either
\begin{enumerate}
    \item $(x_{1}x_{2}\dots x_{n-1})(a_{1}+x_{n}) \in I$, or
    \item $(x_{1}x_{2}\dots x_{n-1})(a_{2}+x_{n+1}) \in I$, or
    \item $(x_{1}x_{2}\dots \reallywidehat{x_{i}} \dots \reallywidehat{x_{j}} \dots x_{n-1})(a_{1}+x_{n})(a_{2}+x_{n+1}) \in I$.
\end{enumerate}
In case 1, after distributing we can see that $x_{1}x_{2}\dots x_{n-1}x_{n} \in I$, and thus $I$ is a graded $n$-absorbing ideal. Likewise, in case 2, after distributing we can see that $x_{1}x_{2}\dots x_{n-1}x_{n+1} \in I$, and again, $I$ is a graded $n$-absorbing ideal. Finally, in case 3, after distributing we can see that
\begin{align*}
    & (x_{1}x_{2}\dots \reallywidehat{x_{i}} \dots \reallywidehat{x_{j}} \dots x_{n-1})(a_{1}+x_{n})(a_{2}+x_{n+1}) \\
    & = (x_{1}x_{2}\dots \reallywidehat{x_{i}} \dots \reallywidehat{x_{j}} \dots x_{n-1})(a_{1}a_{2}+ a_{1}x_{n+1}+a_{2}x_{n}+x_{n}x_{n+1}) \\
    & = (x_{1}x_{2}\dots \reallywidehat{x_{i}} \dots \reallywidehat{x_{j}} \dots x_{n-1})a_{1}a_{2} + (x_{1}x_{2}\dots \reallywidehat{x_{i}} \dots \reallywidehat{x_{j}} \dots x_{n-1})a_{1}x_{n+1} \\
    & + (x_{1}x_{2}\dots \reallywidehat{x_{i}} \dots \reallywidehat{x_{j}} \dots x_{n-1})a_{2}x_{n} + (x_{1}x_{2}\dots \reallywidehat{x_{i}} \dots \reallywidehat{x_{j}} \dots x_{n-1})x_{n}x_{n+1}.
\end{align*}
Since $(x_{1}x_{2}\dots \reallywidehat{x_{i}} \dots \reallywidehat{x_{j}} \dots x_{n-1})x_{n}x_{n+1} \in I$, we again conclude that $I$ is a graded $n$-absorbing ideal. Now we may assume $(x_{1}x_{2}\dots \reallywidehat{x_{i}} \dots \reallywidehat{x_{j}} \dots x_{n}x_{n+1})I_{g}^{2} = (0)$.

Continuing in this manner, each successive step either establishes the graded $n$-absorbing property or forces an additional power of $I_{g}$ to vanish. After at most $n+1$ iterations, one concludes either that $I$ is graded $n$-absorbing or that $I_{g}^{n+1}=0$. The induction formalizes this iterative argument.

\end{proof}

\noindent Theorem \ref{Thm:PrimaryResult1} immediately implies that if a graded ring has no nonzero nilpotent homogeneous elements, then every graded weakly $n$-abosring ideal is graded $n$-absorbing. In particular, this applies to graded polynomial rings over fields. Consequently, throughout the principal monomial setting developed later, it is sufficient to work exclusively with graded $n$-absorbing ideals.

Recall that the graded radical $I$ of $G(R)$, denoted by $\operatorname{Grad}(I)$, is the set of all $x \in R$ such that for each $g \in G$, there exists $n_{g} > 0$ with the property that $x_{g}^{n_{g}} \in I$. 

\begin{corollary}\label{cor:primaryresult1}
Let $\displaystyle I = \bigoplus_{g \in G} I_{g}$ be a graded weakly $n$-absorbing ideal which is not a graded $n$-absorbing ideal of $G(R)$. Then $\operatorname{Grad}(I)=\operatorname{Grad}(0)$.
\end{corollary}
\begin{proof}
It is sufficient to show that $\operatorname{Grad}(I) \subseteq \operatorname{Grad}(0)$ since the reverse inclusion is immediate. Let $x \in I$. $I$ is a graded ideal, so we can write
\[
x=\sum_{g\in G}x_g,
\]
where $x_{g}\in I_g$.  Since $I$ is not graded $n$-absorbing, then by Theorem \ref{Thm:PrimaryResult1}, $I_{g}^{n+1} = (0)$ for all $g \in G$. Thus it follows that since $I$ is graded, each homogeneous component $x_g \in I_g$, so $x_g^{n+1} \in I_g^{n+1} = (0)$.

Thus every homogeneous component $x_g$ belongs to
$\operatorname{Grad}(0)$. Since $\operatorname{Grad}(0)$ is a graded ideal,
it follows that
\[
x\in\operatorname{Grad}(0).
\]
Hence
\[
I\subseteq\operatorname{Grad}(0).
\]
Applying the graded radical operator yields
\[
\operatorname{Grad}(I)
\subseteq
\operatorname{Grad}(\operatorname{Grad}(0))
=
\operatorname{Grad}(0),
\]
completing the proof.

\end{proof} 

Theorem \ref{Thm:PrimaryResult1} and Corollary \ref{cor:primaryresult1} show that the distinction between graded weakly $n$-absorbing ideals and graded $n$-absorbing ideals is entirely determined by the nilpotent behavior of the homogeneous components. 

Indeed, if a graded weakly $n$-absorbing ideal fails to be graded $n$-absorbing, then every homogeneous component satisfies $I_{g}^{n+1}=0$, and consequently its graded radical coincides with the graded radical of the zero ideal. Thus, the distinction between the two notions is completely governed by the nilpotent behavior of the homogeneous components.

\begin{example}[A graded weakly $n$-absorbing ideal that is not $n$-absorbing]
Let $R=F[x_{1},x_{2},\dots,x_{n+1}]/(x_{1}x_{2}\cdots x_{n+1})$ with the standard grading and take $I = (0)$. The zero ideal is homogeneous, so it is a graded ideal. Since there are no nonzero elements in $I$, the weakly $n$-absorbing condition is satisfied vacuously. However, $I$ is not $n$-absorbing because $x_{1}x_{2}\cdots x_{n+1} = 0 \in I$, while every $n$-fold product obtained by omitting one variable is nonzero in $R$.
    
\end{example}

\begin{example}[Graded weakly $n$-absorbing is graded $n$-absorbing ]
Let 
\[
R=F[x_{1},x_{2},\dots,x_{m}]
\]
 and consider the principal monomial ideal $I = (x_{1}x_{2}\cdots x_{n})$. We first observe that $I$ is a graded weakly $n$-absorbing ideal of $R$. 

Let $f_{1},f_{2},\dots,f_{n+1} \in h(R)$ satisfy
\[
f_{1}f_{2}\cdots f_{n+1} \in I.
\]
Then 
\[
x_{1}x_{2}\cdots x_{n} \mid f_{1}f_{2}\cdots f_{n+1}.
\]
Since each variable $x_{j}$ must divide at least one of the homogeneous factors $f_{i}$, the $n$-variables $x_{1},x_{2},\dots,x_{n}$ are distributed among the $n+1$ factors. Since the monomial $x_{1}x_{2}\cdots x_{n}$ contains exactly $n$-distinct variables while there are $n+1$ homogeneous factors, at least one factor contributes none of these required variables. Removing that factor preserves divisibility by $x_{1}x_{2}\cdots x_{n}$, so
\[
f_{1}\cdots \reallywidehat{f_{i}} \cdots f_{n+1} \in I.
\]
Therefore, $I$ is a graded $n$-absorbing ideal. 
\end{example}

The preceding examples illustrate the two alternatives of Theorem \ref{Thm:PrimaryResult1}. The first example shows that graded weakly $n$-absorbing ideals need not be graded $n$-absorbing in the presence of nontrivial zero products. In contrast, the second example demonstrates that this pathology disappears in polynomial rings over fields, where principal monomial ideals satisfy both notions simultaneously. Since the combinatorial model developed later is constructed in this setting, we may work entirely within graded $n$-absorbing ideals. 

\subsection{n-Absorbing Homogeneous Components}

The graded $n$-absorbing condition is determined by the interaction between the multiplicative structure of homogeneous elements and the additive decomposition of the grading. In this section, we introduce the notion of an $n$-absorbing homogeneous component and establish how componentwise $n$-absorbing behavior influences graded weakly $n$-absorbing ideals.

The motivation for this definition comes from the graded $2$-absorbing subgroup framework introduced by Naghani and Moghimi \cite{naghani2016graded}. However, the original formulation does not fully account for the behavior of degrees under multiplication. Indeed, removing one homogeneous factor from a product decreases its degree, so the resulting product generally cannot lie in the same homogeneous component. 

\begin{example}[Counterexample to Definition 2 of \cite{naghani2016graded}]
    Let $R=F[x,y]$ and take $a=x$, $b=y$, and $c=x$. Then $a,b,c \in h(R)$ and $abc=x^{2}y \in R_{3}$. In their original version of Definition 2, we would require one of $ab = xy$, $ac=x^{2}$, or $bc = xy$ to lie in $I_{3}$. But $xy, x^{2} \in R_{2}$, rather than $R_{3}$. Since $I_{3} \subseteq R_{3}$, neither $xy$ nor $x^{2}$ can belong to $I_{3}$. 
\end{example}

Hence the original formulation of Definition 2 cannot be satisfied in this case and the argument used in the proof of Proposition 1 of \cite{naghani2016graded} relies on the same incompatibility.

To remedy this issue, we reformulate the definition so that the absorbing condition is compatible with the grading while preserving the intended local-to-global correspondence between the additive structure and graded $n$-absorbing ideals. 

\begin{definition}\label{def:nabs-component}
Let $R=\bigoplus_{g\in G}R_{g}$ be a graded ring and let $I = \bigoplus_{g \in G} I_{g}$ be a graded ideal of $R$. For $g \in G$, we say that the homogeneous component $I_{g}$ is an \textbf{\textit{n}-absorbing homogeneous component} of $I$ if $I_{g} \neq R_{g}$ and whenever $x_{1},x_{2},\ldots,x_{n+1}\in h(R)$ for $x_{1}x_{2}\cdots x_{n+1} \in I_{g}$, then $x_1x_2\cdots\reallywidehat{x_i}\cdots x_{n+1}\in I$ for some $i \in \{1, 2, \dots, n+1\}$.
\end{definition}

Using our new construction, we now present the new version of Proposition 1 of \cite{naghani2016graded} with updated proof.

\begin{proposition}\label{prop:components}
Let $I=\bigoplus_{g\in G}I_{g}$ be a graded ideal of $R$. If every homogeneous component $I_{g}$ of $I$ is an $n$-absorbing homogeneous component, then $I$ is a graded $n$-absorbing ideal of $R$.
\end{proposition}

\begin{proof}
Let $x_1,x_2,\ldots,x_{n+1}\in h(R)$ be such that $x_{1}x_{2}\cdots x_{n+1}\in I.$ Since each $x_i$ is homogeneous, the product $x_{1}x_{2}\cdots x_{n+1}$ is also homogeneous. Hence there exists a unique $g\in G$ such that $x_1x_2\cdots x_{n+1}\in I_g.$ By hypothesis, $I_{g}$ is an $n$-absorbing homogeneous component of $I$. Therefore, for some $i \in \{ 1, 2, \dots, n+1\}$, we have that $x_1x_2\cdots\reallywidehat{x_i}\cdots x_{n+1}\in I.$ Thus, whenever the product of $n+1$ homogeneous elements belongs to $I$, the product obtained by omitting one of the factors also belongs to $I$. Hence $I$ is a graded $n$-absorbing ideal of $R$.
\end{proof}

\noindent This proposition tells us that if each homogeneous component of $I_{g}$ satisfies the $n$-absorbing homogeneous component condition, then the entire graded ideal automatically inherits the graded $n$-absorbing property. That is, instead of verifying the graded $n$-absorbing condition on the whole ideal $I$, we can instead verify it component-by-component.

\begin{example}
Let $R=F[x_1,\dots,x_n]$ be the standard graded polynomial ring and consider the
principal monomial ideal
\[
I=(x_1x_2\cdots x_n).
\]
For each $d\geq n$, let $I_{d}=I\cap R_{d}$ denote the degree-$d$ homogeneous
component of $I$.

We claim that each homogeneous component $I_d$ is an $n$-absorbing homogeneous component of $I$. Suppose that
\[
f_{1},f_{2},\ldots,f_{n+1}\in h(R)
\]
satisfy
\[
f_{1}f_{2}\cdots f_{n+1}\in I_{d}.
\]
Then
\[
x_1x_2\cdots x_n\mid f_1f_2\cdots f_{n+1}.
\]
Hence, for every $j\in\{1,\ldots,n\}$, the variable $x_j$ divides at least one
of the homogeneous factors $f_i$. Since there are only $n$ variables but
$n+1$ factors, at least one factor can be removed without losing divisibility
by $x_1x_2\cdots x_n$. Therefore, for some $i$,
\[
f_1\cdots\widehat{f_i}\cdots f_{n+1}\in I.
\]
Thus each homogeneous component $I_{d}$ satisfies the $n$-absorbing homogeneous component condition.
\end{example}

\subsection{Quotient Stability and Absorbing Hierarchies}

We extend structural results known for graded 2-absorbing ideals to the graded $n$-absorbing setting. In particular, we investigate quotient stability and the relationship between  graded $n$-absorbing ideals with different absorbing indices.

\subsubsection{Quotients of Graded Weakly n-Absorbing Ideals}
\begin{proposition} \label{prop:quotients}
    Let $R = \bigoplus_{g \in G} R_{g}$ be a graded commutative ring, and let $J \subseteq I$ be proper graded ideals of $R$. Then the following hold:
\begin{enumerate}
    \item If $I$ is a graded weakly $n$-absorbing ideal of $R$, then $I/J$ is a graded weakly $n$-absorbing ideal of $R/J$. 
    \item If $J$ is a graded weakly $n$-absorbing ideal of $R$ and $I/J$ is a graded weakly $n$-absorbing ideal of $R/J$, then $I$ is a graded weakly $n$-absorbing ideal of $R$.
\end{enumerate} 
\end{proposition}

\begin{proof}

\

\begin{enumerate}
    \item Suppose that $I$ is a proper graded weakly $n$-absorbing ideal of $R$. Since $J$ is graded, the quotient $R/J$ inherits a $G$-grading given by $(R/J)_{g} = (R_{g}+J)/J$. Hence every homogeneous element of $R/J$ can be represented by a homogeneous element of $R$. Let $x_{1}+J, x_{2}+J,\dots,x_{n+1}+J$ be homogeneous elements of $R/J$ with representatives $x_{i} \in h(R)$ satisfying  $0\neq (x_1+J)(x_2+J)\cdots(x_{n+1}+J)\in I/J$.  Since the product is nonzero in $R/J$ we have $x_{1}x_{2}\cdots x_{n+1} \notin J$.  Also, $x_{1}x_{2}\cdots x_{n+1} \in I$. Since $I$ is graded weakly $n$-absorbing, there exists some index $i$ such that $x_1\cdots\widehat{x_i}\cdots x_{n+1}\in I$. Therefore, $(x_1+J)\cdots\widehat{(x_i+J)}\cdots(x_{n+1}+J)\in I/J,$ and thus $I/J$ is a graded weakly $n$-absorbing ideal of $R/J$.

    \item Suppose that $J$ is a graded weakly $n$-absorbing ideal of $R$ and $I/J$ is a graded weakly $n$-absorbing ideal of $R/J$. Let $x_{1},x_{2},\dots,x_{n+1} \in h(R)$ such that $0\neq x_1x_2\cdots x_{n+1}\in I$. If $x_{1}x_{2}\cdots x_{n+1} \in J$, then, since $J$ is graded weakly $n$-absorbing, there exists some index $i$ such that $x_1\cdots\widehat{x_i}\cdots x_{n+1}\in J\subseteq I.$ Now suppose that $x_{1}x_{2}\cdots x_{n+1}\notin J.$ Then $0\neq (x_{1}+J)(x_{2}+J)\cdots(x_{n+1}+J)\in I/J.$ Since $I/J$ is graded weakly $n$-absorbing, there exists some index $i$ such that $(x_1+J)\cdots\widehat{(x_i+J)}\cdots(x_{n+1}+J)\in I/J.$ Therefore $x_1\cdots\widehat{x_i}\cdots x_{n+1}\in I,$ and thus $I$ is a graded weakly $n$-absorbing ideal of $R$.
\end{enumerate}
\end{proof}

This proposition establishes quotient stability for graded weakly $n$-absorbing ideals under graded quotient constructions. 

\subsubsection{Monotonicity of Absorbing Indices} In \cite{becker_results_2015}, Becker proved that if $I$ is an $n$-absorbing ideal of a commutative ring $R$, then $I$ is an $m$-absorbing ideal of $R$ for all $m \geq n$. The following theorem shows that Becker's monotonicity result extends to graded $n$-absorbing ideals.

\begin{theorem} \label{thm:mn}
    Let $I$ be a graded $n$-absorbing ideal of $R$. Then $I$ is a graded $m$-absorbing ideal of $R$ for all $m \geq n$.
\end{theorem}

\begin{proof}

    Let $m \geq n$ and suppose that
    \[
    x_{1}x_{2}\cdots x_{m+1} \in I
    \]
    for $x_{1},x_{2},\dots,x_{m+1} \in h(R)$. Set
    \[
    y=x_{n+1}x_{n+2}\cdots x_{m+1}.
    \]
    Since the product of homogeneous elements are homogeneous, we have $y \in h(R).$ Therefore
    \[
    x_{1}x_{2}\cdots x_{n} y \in I.
    \]
    Because $I$ is graded $n$-absorbing, there exists some factor among $x_{1},x_{2},\dots, x_{n},y$ whose removal leaves a product contained in $I$.

    If one of the factors $x_{1},x_{2},\dots,x_{n}$ is removed, then we immediately obtain a product of $m$ homogeneous elements contained in $I$. 

    If the factor $y$ is removed, then 
    \[
    x_{1}x_{2}\cdots x_{n} \in I.
    \]
    Multiplying by the remaining homogeneous factors $x_{n+2},\dots,x_{m+1}$ gives
    \[
    x_{1}x_{2}\cdots x_{n}x_{n+2}\cdots x_{m+1} \in I,
    \]
    which is again a product of $m$ homogeneous elements contained in $I$.

    Therefore, whenever a product of $m+1$ homogeneous elements lies in $I$, there exists a product of $m$ homogeneous elements contained in $I$. Hence $I$ is a graded $m$-absorbing ideal of $R$.    
\end{proof}

The monotonicity property shows that the classes of graded $n$-absorbing ideals are not unrelated collections indexed by $n$. Rather, the classes form a nested hierarchy:
\begin{equation*}
   \mathcal{A}_{1} \subseteq \mathcal{A}_{2} \subseteq \dots \subseteq \mathcal{A}_{n} \subseteq \dots,
\end{equation*}
where $\mathcal{A}_{n}$ denotes the collection of graded $n$-absorbing ideals. Consequently, every graded $n$ absorbing ideal remains graded $m$ absorbing for every $m \geq n$. Thus the classes of graded $n$-absorbing ideals form an ascending hierarchy indexed by $n$, where increasing the absorbing index enlarges the class of ideals satisfying the condition. In the next section, this algebraic hierarchy admits a geometric interpretation for principal monomial ideals, where the parameter $n$ becomes a bound on exponent vectors defining a lattice simplex.

\subsection{Invariance and Transfer Properties} 

Having established several intrinsic structural properties of graded $n$-absorbing ideals, we now examine how these properties behave under graded ring homomorphisms. We first show that graded weakly $n$-absorbing ideals are invariant under graded ring isomorphisms and then establish preservation results for graded $n$-absorbing primary ideals under inverse images and suitable homomorphic images.

\subsubsection{Graded Isomorphisms}
 The following proposition shows that the graded weakly $n$-absorbing property is invariant under graded ring isomorphisms.

\begin{proposition} \label{prop:isoinvariance}
    Let $ R=\bigoplus_{g\in G}R_{g}$ and $S=\bigoplus_{g\in G}S_{g}$ be graded rings, and let $\varphi : R \rightarrow S$ be a graded ring isomorphism satisfying $\varphi(R_g)=S_g$ for every $g \in G$. If $I \subseteq R$ is a graded weakly $n$-absorbing ideal, then $\varphi(I)$ is a graded weakly $n$-absorbing ideal of $S$.
\end{proposition}

\begin{proof}
Let $y_{1},y_{2},\dots,y_{n+1} \in h(S)$ such that $0 \neq y_{1}y_{2}\cdots y_{n+1} \in \varphi(I)$. Since $\varphi$ is a graded ring isomorphism, its inverse $\varphi^{-1}: S \rightarrow R$ is also a graded ring isomorphism. Hence, $x_i=\varphi^{-1}(y_i)\in h(R)$ for each $i=1,2,\dots,n+1$. Since $y_{1} \cdots y_{n+1} \in \varphi(I)$, applying $\varphi^{-1}$ gives  $x_1x_2\cdots x_{n+1}\in I.$ Moreover, since \(\varphi\) is injective and $y_{1}y_{2}\cdots y_{n+1}\neq 0,$ it follows that $x_{1}x_{2}\cdots x_{n+1} \neq 0$.

Since $I$ is a graded weakly $n$-absorbing ideal of  $R$, there exists some  $i\in\{1,2,\dots,n+1\}$ such that
$x_1x_2\cdots\widehat{x_i}\cdots x_{n+1}\in I.$ Applying $\varphi$, we obtain
\[
y_1y_2\cdots\widehat{y_i}\cdots y_{n+1}
=
\varphi(x_1x_2\cdots\widehat{x_i}\cdots x_{n+1})
\in \varphi(I).
\]
Therefore, \(\varphi(I)\) is a graded weakly $n$-absorbing ideal of $S$.
\end{proof}

\subsubsection{Graded Homomorphisms and Primary Ideals}

We extend the preceding notions of $n$-absorbing primary and weakly $n$-absorbing primary ideals to the setting of graded rings as follows.

\begin{definition}
    Let $G(R)$ be a $G$-graded ring, and let $I$ be a graded ideal of $G(R)$. Then $I$ is a \textbf{graded n-absorbing primary ideal of $G(R)$}, if $I \neq R$, and whenever $x_1, x_2, \dots, x_{n+1} \in h(R)$ and $x_{1}x_{2}\cdots x_{n+1} \in I$, then either $x_{1}x_{2}\cdots x_{n} \in I $, or $x_{1}\cdots \reallywidehat{x_{i}} \cdots x_{n+1} \in \sqrt{I}$ for some $i \in [1,n]$.
    
\end{definition}  

\begin{definition}
    Let $G(R)$ be a $G$-graded ring, and let $I$ be a graded ideal of $G(R)$. Then $I$ is a \textbf{graded weakly \textit{n}-absorbing primary ideal} of $G(R)$, if $I \neq R$, and whenever $x_1, x_2, \dots, x_{n+1} \in h(R)$ and $ 0 \neq x_{1}x_{2}\cdots x_{n+1} \in I$, then either $x_{1}x_{2}\cdots x_{n} \in I $ , or $x_{1}\cdots \reallywidehat{x_{i}} \cdots x_{n+1} \in \sqrt{I}$ for some $i \in [1, n]$.
    
\end{definition} 

Proposition \ref{prop:isoinvariance} establishes invariance under graded isomorphisms. The next theorem describes the behavior of graded $n$-absorbing primary ideals under graded ring homomorphisms. 

\begin{theorem}
    Let $f:R\rightarrow R'$ be a $G$-graded ring homomorphism of $G$-graded commutative rings. Then 
    \begin{enumerate}
        \item If $I^{\prime}$ is a graded $n$-absorbing primary ideal of $R^{\prime}$, then $f^{-1}(I^{\prime})$ is a graded $n$-absorbing primary ideal of $R$. 
        \item If $f$ is a surjective and $I$ is a graded $n$-absorbing primary ideal of $R$ containing $\ker(f)$, then $f(I)$ is a graded $n$-absorbing primary ideal of $R^{\prime}$. 
    \end{enumerate}
\end{theorem}

    \begin{proof}

    \
    
    \begin{enumerate}
        \item Let $I^{\prime}$ be a graded $n$-absorbing primary ideal of $R^{\prime}$. Let $x_1, x_2, \dots, x_{n+1} \in h(R)$ such that $x_1\dots x_{n+1} \in f^{-1}(I^{\prime})$. Then, $f(x_1)\dots f(x_{n+1}) \in I^{\prime}.$ Now, $I^{\prime}$ is a graded $n$-absorbing primary ideal of $R^{\prime}$, so either 
      \begin{enumerate}
          \item $f(x_1)\dots f(x_n) \in I^{\prime}$, or
          \item $f(x_1)\cdots \reallywidehat{f(x_i)} \cdots f(x_{n+1}) \in \sqrt{I^{\prime}}$ for some $i \in [1,n]$. 
      \end{enumerate}
      In the first case, $x_1 \cdots x_n \in f^{-1}(I^{\prime})$, while in the second case, $x_1 \cdots \reallywidehat{x_i} \cdots x_{n+1} \in f^{-1}(\sqrt{I^{\prime}}) = \sqrt{f^{-1}(I^{\prime})}$. Therefore $f^{-1}(I^{\prime})$ is a graded $n$-absorbing primary ideal of $R$. 

     \item Suppose
\[
x_1',\dots,x_{n+1}' \in h(R')
\]
such that
\[
x_1'\cdots x_{n+1}' \in f(I).
\]

Since $f$ is a surjective graded homomorphism and $I$ is a graded
$n$-absorbing primary ideal of $R$, there exist
$x_1,\dots,x_{n+1}\in h(R)$ such that
\[
f(x_i)=x_i'
\quad \text{for all } i=1,\dots,n+1.
\]

Then
\begin{align*}
f(x_1\cdots x_{n+1})
    &= f(x_1)\cdots f(x_{n+1}) \\
    &= x_1'\cdots x_{n+1}'
    \in f(I).
\end{align*}

Hence,
\[
x_1\cdots x_{n+1}
    \in f^{-1}(f(I))
    = I+\ker(f)
    = I,
\]
since $\ker(f)\subseteq I$.

Because $I$ is a graded $n$-absorbing primary ideal, either
\[
x_1\cdots x_n\in I,
\]
or
\[
x_1\cdots\widehat{x_i}\cdots x_{n+1}
    \in \sqrt{I}
\]
for some $i\in\{1,\dots,n\}$.

In the first case,
\begin{align*}
f(x_1\cdots x_n)
    &= f(x_1)\cdots f(x_n) \\
    &= x_1'\cdots x_n'
    \in f(I).
\end{align*}

In the second case,
\begin{align*}
f(x_1\cdots \widehat{x_i}\cdots x_{n+1})
    &= f(x_1)\cdots \widehat{f(x_i)}
       \cdots f(x_{n+1}) \\
    &= x_1'\cdots \widehat{x_i'}
       \cdots x_{n+1}' \\
    &\in f(\sqrt{I})
    \subseteq \sqrt{f(I)}.
\end{align*}

Thus $f(I)$ is a graded $n$-absorbing primary ideal of $R'$.
    \end{enumerate}
      
    \end{proof}

This result shows that graded $n$-absorbing primary ideals are stable under the standard constructions connecting graded rings. In particular, quotient graded rings inherit graded $n$-absorbing primary ideals from ideals containing the defining homogeneous kernel, while graded homomorphisms allow these structures to be transported to larger graded settings via inverse images.

Collectively, these results demonstrate that the structural properties underlying the theory of graded 2-absorbing ideals extend naturally to arbitrary $n$. Having established this general framework, we now specialize to principal monomial ideals, where the theory admits a remarkably concrete combinatorial realization.

\section{The Combinatorial Model for Principal Monomial Ideals}

Anderson and Badawi pose three conjectures on $n$-absorbing ideals in their 2011 paper \cite{anderson_n_2011}. The first, relating to the Noether exponent of $\sqrt{I}$, was resolved in \cite{choi2019n}. The second conjecture asserts that every $n$-absorbing ideal is strongly $n$-absorbing, and was recently proved by Secord \cite{secord24}. Conjecture III remains open in general and has been verified only for restricted classes of rings \cite{anderson_n_2011,choi2022n,laradji2017n,sihem2017anderson}. It states:

\

\noindent \textbf{Conjecture III: }\textit{
If $I$ is an $n$-absorbing ideal of a commutative ring $R$, then
$I[X]$ is an $n$-absorbing ideal of $R[X]$.}

\

We first verify Conjecture III in the setting of principal monomial ideals. This provides a polynomial extension invariance result for the absorbing index of this class of ideals. Having established this invariance, we turn to the internal structure of principal monomial ideals and develop a combinatorial model describing their graded absorbing behavior.

The monotonicity results from the previous section suggest that graded $n$-absorbing ideals should admit a natural organization according to their absorbing index. For a principal monomial ideal generated by
\[
x^\alpha=x_1^{a_1}\cdots x_m^{a_m},
\]
the exponent vector
\[
\alpha=(a_1,\dots,a_m)\in\mathbb{N}^{m},
\]
completely determines the generator, while the absorbing index depends only on its total degree
\[
\vert \alpha \vert=a_{1}+\cdots+a_{m}.
\]
Consequently, the graded $n$-absorbing principal monomial ideals correspond precisely to the lattice points contained in the nested family
\[
\Delta_{1}^{m} \subseteq \Delta_{2}^{m} \subseteq \dots \subseteq \Delta_{n}^{m} \subseteq \dots
\]
where
\[
\Delta_{n}^{m}=\{\alpha\in\mathbb{N}^m:|\alpha|\leq n\},
\]
where each $\Delta_{n}^{m}$ consists of all exponent vectors of total degree at most $n$.

\subsection{Polynomial Extension Invariance}
 Following notation from \cite{choi2019n}, we denote the smallest integer $n$ such that ideal $I$ is $n$-absorbing, 
\[\omega(I)=\textnormal{min}\{n \in \mathbb{N} | I \textnormal{ is $n$-absorbing}\}. \] 

\noindent We can equivalently pose Conjecture III as $\omega_{R[X]}(I[X]) = \omega_{R}(I).$ We restrict to principal monomial ideals and show this conjecture holds.

\begin{theorem} \label{thm:conjectureiii}
    Let $F$ be a field and $R = F[x_1, \dots, x_m]$. For every proper, principal monomial ideal $I = (x^{\alpha})$ of $R$, $$\omega_{R[X]}(I[X]) = \omega_{R}(I).$$
\end{theorem}
\begin{proof}
    Let $I = (x^{\alpha})$, and define $X$ such that $X \notin \{x_1, \dots, x_m\}$. Then the extension $I[X]$ is the ideal of $R[X] = F[x_1, \dots, x_m, X]$ generated by $(x^{\alpha})$. That is, $I[X] = (x^{\alpha}) \cdot R[X] $, which is still a principal monomial ideal (but in the polynomial ring $F[x_1, \dots, x_m, X]$ over the field $F$ in $m+1$ variables, generated by the monomial $x^{\alpha}\cdot X^0$). By Corollary 3.7 in Choi and Walker \cite{choi2019n} applied to $R[X] = F[x_1, \dots, x_m, X]$, we have \[ \omega_{R[X]}(I[X]) = \text{deg}_{R[X]}(x^{\alpha}) = |\alpha| + 0 = |\alpha|.\] Similarly we see that \[ \omega_{R}(I) = \text{deg}_R(x^{\alpha}) = |\alpha|.\] Thus we have that $\omega_{R[X]}(I[X]) = |\alpha|=\omega_{R}(I).$
\end{proof}

\begin{corollary}
Let $F$ be a field, let $R=F[x_1,\ldots,x_m]$, and let $I=(x^\alpha)$ be a proper principal monomial ideal of $R$. Then, for every integer $t\geq 1$, $I$ is $n$-absorbing if, and only if, $I[Y_{1},\dots,Y_{t}]$ is $n$-absorbing. In particular
\[
\omega_{R[Y_1,\ldots,Y_t]}(I[Y_1,\ldots,Y_t])
=
\omega_R(I).
\]
\end{corollary}
\begin{proof}
    By repeated application of Theorem \ref{thm:conjectureiii},
    \[
    \omega_{R[Y_{1},\dots,Y_{t}]}(I[Y_{1},\dots,Y_{t}]) = \omega_{R}(I).
    \]
    Since an ideal is $n$-absorbing if and only if
    \[
    \omega(I) \leq n,
    \]
    it follows that $I$ is $n$-absorbing if and only if $I[Y_{1},\dots,Y_{t}]$ is $n$-absorbing.
\end{proof}

Thus, polynomial extension does not alter the absorbing index of a principal monomial ideal. This invariance allows us to focus on the intrinsic combinatorial structure of principal monomial ideals, which we now describe through their exponent vectors. 

\subsection{Principal Monomial Ideals as Lattice Points} 

 Let $R = F[x_{1},x_{2},\dots,x_{m}]$ be a polynomial ring, where $F$ is  a field, and let $I = (x_{1}^{a_{1}}x_{2}^{a_{2}} \dots x_{m}^{a_{m}})$ be a monomial ideal. Define $\vert \alpha \vert = a_{1}+a_{2}+\dots+a_{m}$ as the total degree of the monomial generator $x^{\alpha}$. Additionally, $\vert \alpha \vert$ corresponds to the total number of irreducible factors counted with multiplicity. We first develop a canonical bijection of the divisibility postet  of the monoid of monomials in $\mathbb{Z}[x,y]$ to the lattice $(\mathbb{N}^{2},\leq)$. We then generalize this result to the full $\mathbb{N}^{m}$ lattice.

 Consider the graded ring $R = \mathbb{Z}[x,y]$ with the standard grading unless indicated otherwise. We begin by studying the relationship of the principal monomial ideals $I = (x^{a}y^{b})$, $a,b \in \mathbb{N}$, of $R$. As illustrated in Figure \ref{fig:HasseIso1}, we can create an isomorphism between the Hasse diagram of monomials and the standard lattice graph by imposing the rules:
\begin{enumerate}
    \item Moving right in the $\mathbb{N}^{2}$ lattice increases the powers of $x$ in the Hasse diagram.
    \item Moving upward in the $\mathbb{N}^{2}$ lattice increases the powers of $y$ in the Hasse diagram.
\end{enumerate}

 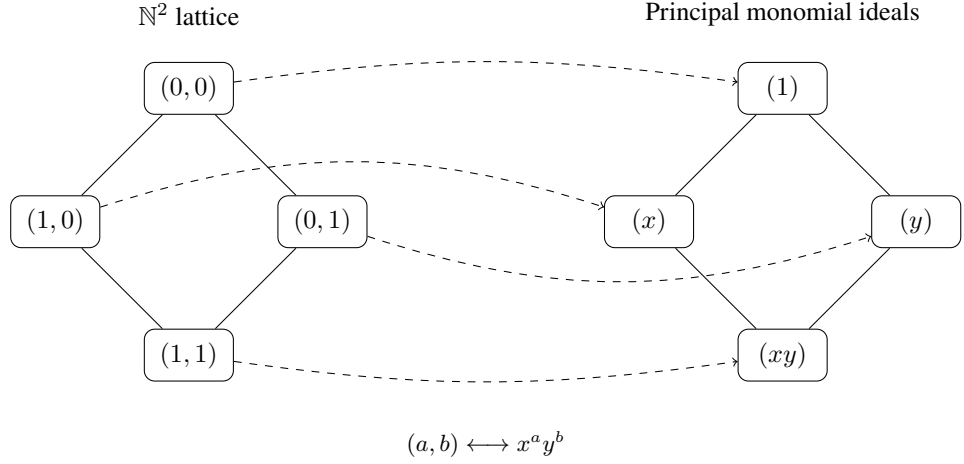
\begin{figure}[ht]
\centering

\resizebox{\linewidth}{!}{%
\begin{tikzpicture}[
    node distance=1.8cm and 2.2cm,
    every node/.style={
        draw,
        rounded corners,
        minimum width=1.2cm,
        minimum height=0.7cm,
        align=center
    }
]

%------------------------------------------------
% Left diagram: lattice points in N^2
%------------------------------------------------

\node (a00) at (0,0) {$(0,0)$};

\node (a10) at (-1.8,-1.8) {$(1,0)$};
\node (a01) at (1.8,-1.8) {$(0,1)$};

\node (a11) at (0,-3.6) {$(1,1)$};

\draw (a00)--(a10);
\draw (a00)--(a01);
\draw (a10)--(a11);
\draw (a01)--(a11);

\node[draw=none] at (0,1) {$\mathbb{N}^2$ lattice};

%------------------------------------------------
% Right diagram: monomial ideals
%------------------------------------------------

\node (i1) at (8,0) {$(1)$};

\node (ix) at (6.2,-1.8) {$(x)$};
\node (iy) at (9.8,-1.8) {$(y)$};

\node (ixy) at (8,-3.6) {$(xy)$};

\draw (i1)--(ix);
\draw (i1)--(iy);
\draw (ix)--(ixy);
\draw (iy)--(ixy);

\node[draw=none] at (8,1) {Principal monomial ideals};

%------------------------------------------------
% Correspondence arrows
%------------------------------------------------

\draw[dashed,->,bend left=8]  (a00) to (i1);

\draw[dashed,->,bend left=18] (a10) to (ix);

\draw[dashed,->,bend right=18] (a01) to (iy);

\draw[dashed,->,bend right=8] (a11) to (ixy);

\node[draw=none,font=\small] at (4,-4.8)
{$(a,b)\longleftrightarrow x^a y^b$};

\end{tikzpicture}

} % end resizebox
\caption{The canonical isomorphism between $(\mathbb{N}^{2},\leq)$ where $(a,b) \leq (c,d)$ if and only if $a \leq c$ and $b \leq d$ and the divisbility poset $(M,\vert)$ where $M=\{x^{a}y^{b} : (a,b)\in\mathbb{N}^{2}\}$ is the monoid of monomials in $\mathbb{Z}[x,y]$ .}
\label{fig:HasseIso1}

\end{figure}
This leads to the following result relating the Hasse diagrams of monomial ideals of $\mathbb{Z}[x,y]$ ordered by divsibility and the $\mathbb{N}^{2}$ lattice.

\begin{proposition}
    Let $M=\{x^{a}y^{b} : (a,b)\in\mathbb{N}^{2}\}$ be the monoid of monomials in $\mathbb{Z}[x,y]$ ordered by divisibility. Then the divisbility poset $(M,\vert)$ is canonically isomoprhic to the lattice $(\mathbb{N}^{2},\leq)$, where $(a,b) \leq (c,d)$ if and only if $a \leq c$ and $b \leq d$. Under this identification, the Hasse diagram of monomials is exactly the standard lattice graph on $\mathbb{N}^{2}$ with cover relations $(a,b) \prec (a+1,b) $ and $(a,b) \prec (a,b+1) $ corresponding to multiplication by $x$ and $y$ respectively.
\end{proposition}

\begin{proof}
    The correspondence $(a,b) \rightarrow x^{a}y^{b}$ is bijective, and $x^{a}y^{b} \mid x^{c}y^{d}$ if and only if $a \leq c$ and $b \leq d$. The cover relations are therefore exactly multiplication by $x$ or $y$, corresponding to the edges of the lattice graph.
\end{proof}

Progressing from the $\mathbb{N}^{2}$ lattice to the $\mathbb{N}^{m}$ lattice for principal monomial ideals $I=(x^{\alpha})$ for polynomial rings $R=F[x_{1},x_{2},\dots,x_{m}]$ follows a natural pattern. Given a vertex $(a_{1},a_{2},\dots_{a_{m}}) \in \mathbb{N}^{m}$, we create a correspondence with a principal monomial ideal $(x_{1}^{a_{1}}x_{2}^{a_{2}} \cdots x_{m}^{a_{m}}) = (x^{\alpha}$). We now state and prove the full lattice point correspondence result in $\mathbb{N}^{m}$.

\begin{theorem}\label{thm:nmlatticehasse}
Let $R=F[x_1,x_2,\dots,x_m]$ and let
\[
M=\{x^\alpha:\alpha\in\mathbb N^{m}\}
\]
be the monoid of monomials. The map
\[
\alpha\longmapsto x^{\alpha}
\]
is a canonical poset isomorphism
\[
(\mathbb N^m,\leq)\cong(M,\mid),
\]
where $\alpha\leq\beta$ if and only if $\alpha_i\leq\beta_i$ for every coordinate. Moreover, the map
\[
x^\alpha\longmapsto(x^\alpha)
\]
identifies $(M,\mid)$ with the poset of principal monomial ideals ordered by reverse inclusion. Consequently, the Hasse diagram of principal monomial ideals ordered by reverse inclusion is isomorphic to the lattice graph on $\mathbb N^m$.
\end{theorem}

\begin{proof}
     Let $R = F[x_{1},x_{2},\dots,x_{m}]$ be equipped with the standard $\mathbb{N}^{m}$ grading where $\deg(x_{i})=e_{i}$, that is, the grading obtained where you assign each variable its corresponding standard basis vector. Thus, we can express $R = \bigoplus_{\alpha \in \mathbb{N}^{m}} R_{\alpha}$ with $\deg(x_{i})=e_{i}$. Let $I = (x^{\alpha})$ represent the principal monomial ideal generated by $x_{1}^{a_{1}}x_{2}^{a_{2}}\dots x_{m}^{a_{m}}$, where $\alpha = (a_{1},a_{2},\dots,a_{m}) \in \mathbb{N}^{m}$. 

    Define a mapping $\Phi: \mathbb{N}^{m} \to \{\textnormal{Principal monomial ideals of \textit{R}}\}$ by $\Phi (\alpha) = (x^{\alpha}).$ Clearly $\Phi$ is a bijection. Thus, we can relate every algebraic statement about principal monomial ideals into an equivalent statement about exponent vectors in $\mathbb{N}^{m}$. For every $\alpha,\beta \in \mathbb{N}^{m}$, $x^{\alpha} \vert x^{\beta}$ if and only if $a_{i} \leq b_{i}$ for all $i=1,2,\dots,m$ where $\beta = (b_{1},b_{2},\dots,b_{m})$. Hence $(x^{\beta}) \subseteq (x^{\alpha})$ in the set of principal monomial ideals of $R$ if and only if $\alpha \leq \beta$ coordinatewise. Thus the poset of principal monomial ideals ordered by reverse inclusion is isomorphic to the standard product order on $\mathbb{N}^{m}$.

     Since the Hasse diagram records cover relations, it follows that $(x^{\beta}) \subsetneq (x^{\alpha})$ is a cover if and only if there is no principal monomial between them. In coordinates, this means $\beta = \alpha + e_{i}$ for some standard basis vector $e_{i}$. We can see this to be true since
    \begin{enumerate}
        \item[\textbullet] if $\vert \beta - \alpha \vert \geq 2$, then some intermediate vector $\gamma$ exists where $\beta = \alpha + \gamma$;
        \item[\textbullet] if $\vert \beta - \alpha \vert = 1$, that is $\beta - \alpha = e_{i}$, then no intermediate vector exists. 
    \end{enumerate}
    Therefore the edges of the Hasse diagram are exactly identified with the relation $\alpha \leftrightarrow \alpha + e_{i}$. This is exactly the lattice graph on $\mathbb{N}^{m}$.

\end{proof}

We can further extend this framework into a complete combinatorial model for graded $n$-absorbing ideals in the monomial setting by relating the absorbing properties with simplexes in $\mathbb{N}^{m}$. This will be accomplished by utilizing the exponent vector $\alpha$. 

\subsection{The Simplex Characterization of Graded n-Absorbing Ideals} %Centerpiece of our paper

The exponent lattice admits two complementary interpretations. The Hasse diagram records the monomial divisibility structure, while the simplex representation identifies the region of exponent vectors corresponding to graded $n$-absorbing principal monomial ideals. Returning to $R= k[x,y]$, we can expand Figure \ref{fig:HasseIso1} to make the following observation about our existing bijection. For the principal monomial ideals, $(x^{a}y^{b})$ is graded $n$-absorbing if and only if $a+b \leq n$. For example, graded 2-absorbing ideals correspond to a lattice of points in $\mathbb{N}^{2}$ satisfying $a+b\leq2$. Geometrically, this forms the triangle in $\mathbb{N}^{2}$ as illustrated in Figure \ref{fig:HasseIso2}. See Figure \ref{fig:principalmonomialideals3} for the visualization of graded $n$-absorbing principal monomial ideals for increasing values of $n$.

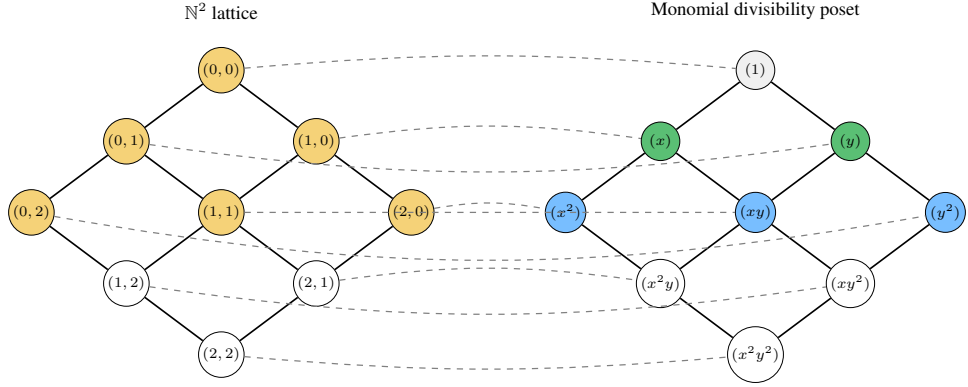
\begin{figure}[ht]
\centering

\resizebox{\textwidth}{!}{%
\begin{tikzpicture}[
    every node/.style={
        draw,
        circle,
        minimum size=6.5mm,
        inner sep=1pt,
        font=\scriptsize
    },
    line/.style={thick},
    corr/.style={dashed, gray, semithick}
]

\definecolor{ab2}{RGB}{245,210,120}
\definecolor{grade2}{RGB}{120,190,255}
\definecolor{gradep}{RGB} {90,191,116}
\definecolor{base}{RGB}{240,240,240}

%%%%%%%%%%%%%%%%%%%%%%%%%%%%%%%%%%%%%%%%%%%%
% LEFT: N^2 lattice (correct product order)
%%%%%%%%%%%%%%%%%%%%%%%%%%%%%%%%%%%%%%%%%%%%

\begin{scope}[xshift=0cm,yshift=0cm]

\node[fill=ab2] (00) at (0,0) {$(0,0)$};
\node[fill=ab2] (10) at (1.6,-1.2) {$(1,0)$};
\node[fill=ab2] (01) at (-1.6,-1.2) {$(0,1)$};

\node[fill=ab2] (20) at (3.2,-2.4) {$(2,0)$};
\node[fill=ab2] (11) at (0,-2.4) {$(1,1)$};
\node[fill=ab2] (02) at (-3.2,-2.4) {$(0,2)$};

\node (21) at (1.6,-3.6) {$(2,1)$};
\node (12) at (-1.6,-3.6) {$(1,2)$};

\node (22) at (0,-4.8) {$(2,2)$};

\draw[line] (00)--(10);
\draw[line] (00)--(01);

\draw[line] (10)--(20);
\draw[line] (10)--(11);

\draw[line] (01)--(11);
\draw[line] (01)--(02);

\draw[line] (20)--(21);
\draw[line] (11)--(21);

\draw[line] (11)--(12);
\draw[line] (02)--(12);

\draw[line] (21)--(22);
\draw[line] (12)--(22);

\end{scope}

%%%%%%%%%%%%%%%%%%%%%%%%%%%%%%%%%%%%%%%%%%%%
% RIGHT: monomial divisibility poset
%%%%%%%%%%%%%%%%%%%%%%%%%%%%%%%%%%%%%%%%%%%%

\begin{scope}[xshift=9cm,yshift=0cm]

\node[fill=base] (r) at (0,0) {$(1)$};

\node[fill=gradep] (x) at (-1.6,-1.2) {$(x)$};
\node[fill=gradep] (y) at (1.6,-1.2) {$(y)$};

\node[fill=grade2] (x2) at (-3.2,-2.4) {$(x^2)$};
\node[fill=grade2] (xy) at (0,-2.4) {$(xy)$};
\node[fill=grade2] (y2) at (3.2,-2.4) {$(y^2)$};

\node (x2y) at (-1.6,-3.6) {$(x^2y)$};
\node (xy2) at (1.6,-3.6) {$(xy^2)$};

\node (x2y2) at (0,-4.8) {$(x^2y^2)$};

\draw[line] (r)--(x);
\draw[line] (r)--(y);

\draw[line] (x)--(x2);
\draw[line] (x)--(xy);

\draw[line] (y)--(xy);
\draw[line] (y)--(y2);

\draw[line] (x2)--(x2y);
\draw[line] (xy)--(x2y);

\draw[line] (xy)--(xy2);
\draw[line] (y2)--(xy2);

\draw[line] (x2y)--(x2y2);
\draw[line] (xy2)--(x2y2);

\end{scope}

%%%%%%%%%%%%%%%%%%%%%%%%%%%%%%%%%%%%%%%%%%%%
% CORRESPONDENCE
%%%%%%%%%%%%%%%%%%%%%%%%%%%%%%%%%%%%%%%%%%%%

\draw[corr,bend left=5] (00) to (r);

\draw[corr,bend left=8] (10) to (x);
\draw[corr,bend right=8] (01) to (y);

\draw[corr,bend left=10] (20) to (x2);
\draw[corr] (11) -- (xy);
\draw[corr,bend right=10] (02) to (y2);

\draw[corr,bend left=8] (21) to (x2y);
\draw[corr,bend right=8] (12) to (xy2);

\draw[corr,bend right=5] (22) to (x2y2);

%%%%%%%%%%%%%%%%%%%%%%%%%%%%%%%%%%%%%%%%%%%%
% LABELS
%%%%%%%%%%%%%%%%%%%%%%%%%%%%%%%%%%%%%%%%%%%%

\node[draw=none,font=\small] at (0,1) {$\mathbb{N}^2$ lattice};
\node[draw=none,font=\small] at (9,1) {Monomial divisibility poset};

%\node[draw=none,font=\small] at (4.5,-6.2)
%{$(a,b)\longleftrightarrow x^a y^b$};

\end{tikzpicture}%
}

\caption{The correspondence between exponent vectors in $\mathbb{N}^{2}$ and principal monomial ideals in the standard graded polynomial ring $k[x,y]$. On the left, the lattice points satisfying $a+b\leq 2$ form the discrete simplex $\Delta_{2}^{2}$. Under the bijection $(a,b)\leftrightarrow (x^{a}y^{b})$, these points correspond to the principal monomial ideals that are graded $2$-absorbing. The ideals shaded light blue are graded $2$-absorbing but not graded prime, while the ideals shaded dark green are graded prime (and therefore trivially graded $2$-absorbing). The origin corresponds to the unit ideal $(1)$, which is excluded from Theorem \ref{thm3-bijection} since only proper ideals are considered.}
\label{fig:HasseIso2}
\end{figure}

The fundamental observation is that the exponent vector $\alpha = (a_{1},a_{2},\dots,a_{m})$ completely determines the ideal and thus classifying the graded $n$-absorbing ideals in this setting becomes a combinatorial problem on lattice points in $\mathbb{N}^{m}$. By Theorem \ref{thm:nmlatticehasse}, we identify principal monomial ideals with lattice points of $\mathbb{N}^{m}$. We now characterize which lattice points correspond to the graded $n$-absorbing ideals.

\begin{theorem}
  Let $F$ be a field and let $R=F[x_{1},x_{2},\dots,x_{m}]$ with the standard grading. Let $I = (x^{\alpha})$ be a proper principal monomial ideal. Then the following are equivalent:
    \begin{enumerate}
        \item $I$ is graded $n$-absorbing.
        \item The exponent vector satisfies $\vert \alpha \vert \leq n$.
        \item The vector $\alpha$ lies in the discrete simplex
        \[
        \Delta_{n}^{m} = \{\alpha \in \mathbb{N}^{m}: \vert \alpha \vert \leq n\}.
        \]
    \end{enumerate}
    \label{thm3-bijection}
\end{theorem}

\begin{figure}[ht]
    \centering
    \includegraphics[width=\textwidth]{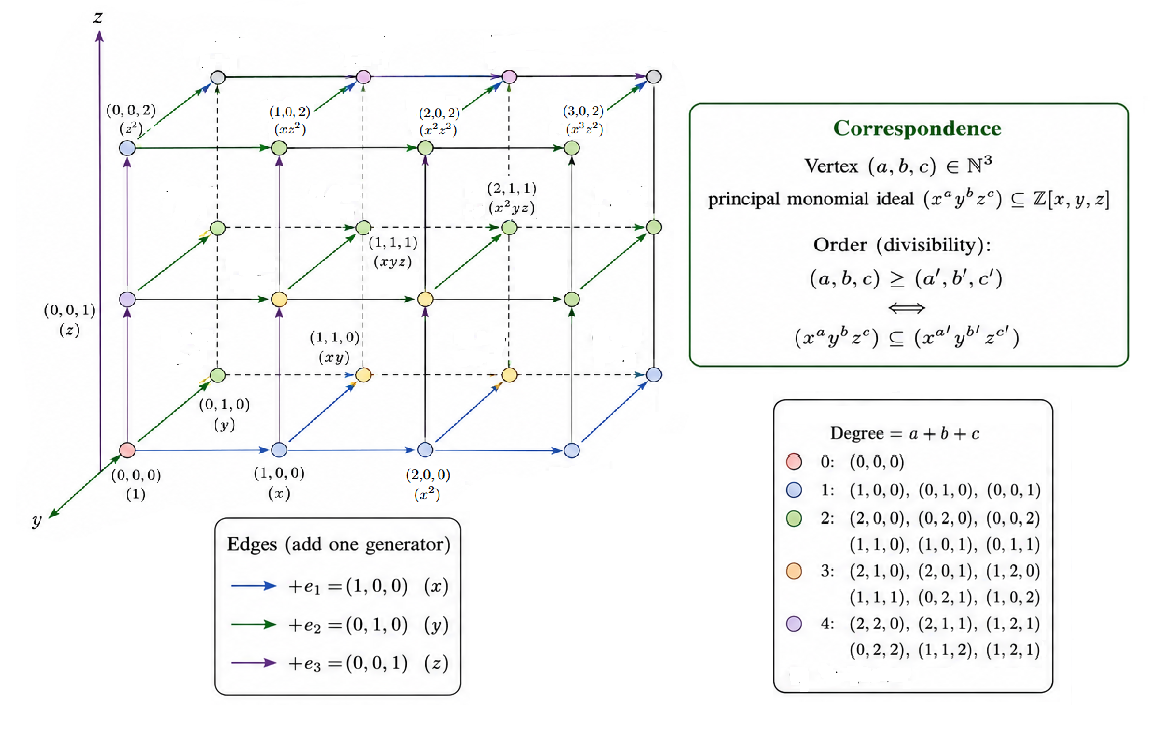}
    \caption{Principal monomial ideals in $\mathbb{Z}[x,y,z]$ and their corresponding exponent vectors. Aside from the origin, which corresponds to the unit ideal (1), the graded $n$-absorbing principal monomial ideals are precisely those whose exponent vectors lie in the discrete simplex $\Delta_{n}^{3} = \{\alpha \in \mathbb{N}^{3}: \vert \alpha \vert \leq n$\}.}
    \label{fig:principalmonomialideals3}
\end{figure}

\begin{proof}
By Theorem \ref{thm:nmlatticehasse}, every principal monomial ideal $I = (x)^{\alpha}$ is identified uniquely with its exponent vector $\alpha \in \mathbb{N}^{m}$, where divisibility corresponds to the coordinatewise partial order. Thus it suffices to determine which exponent vectors correspond to graded $n$-absorbing principal monomial ideals.

Proof of $(1) \implies (2)$: Suppose $I$ is graded $n$-absorbing, and assume, for purposes of contradiction, that $\vert \alpha \vert = k > n$. Then we have 
    \begin{align*}
        x^{\alpha} &= x_{1}^{a_{1}}x_{2}^{a_{2}}\dots x_{m}^{a_{m}} \\
        &= \underbrace{x_1\cdots x_1}_{a_{1}}
           \underbrace{x_2\cdots x_2}_{a_{2}}
            \cdots
            \underbrace{x_m\cdots x_m}_{a_{m}}.
    \end{align*}
    Define $n+1$ homogeneous elements by taking the first $n$ individual variable factors of $x^{\alpha}$ as $u_1, \dots, u_n$ and let $u_{n+1}$ be the products of the remaining $k-n$ factors. Then $u_1u_2\cdots u_{n+1} = x^{\alpha} \in (x^{\alpha})$. We show no $n$-fold subproduct lies in the ideal $(x^{\alpha})$. 
    For $j \leq n:$ the subproduct $u_1u_2\cdots \hat{u_j}\cdots u_{n+1}$ is a monomial of total degree $k-1 < k = |\alpha|$, and thus is not divisible by $x^{\alpha}$, and therefore not in $(x^{\alpha})$. 
    For $j = n+1$: the subproduct $u_1u_2\cdots u_n$ is a monomial of total degree $n<k = |\alpha|$, thus it is also not in $(x^{\alpha})$. In both cases, the $n$-fold subproduct is not in $(x^{\alpha})$, so the ideal is not graded $n$-absorbing. Hence $I$ cannot be graded $n$-absorbing, a contradiction. Thus $\vert \alpha \vert \leq n$. 

Proof of $(2) \implies (1)$: Suppose $\vert \alpha \vert \leq n$. By Corollary 3.7 of Choi and Walker \cite{choi2019n}, $\omega((x^{\alpha})) = \vert \alpha \vert$. Hence $\omega((x^{\alpha})) \leq n$, so $(x^{\alpha})$ is $n$-absorbing ideal of $R$. By Lemma \ref{lem:ntogradedn}, it follows that $(x^{\alpha})$ is graded $n$-absorbing. 

Proof of $(2) \iff (3)$: This equivalence follows immediately from the definition of the discrete
simplex
\[
\Delta_n^m=\{\alpha\in\mathbb N^m:|\alpha|\leq n\}.
\]
Indeed, the condition $\vert\alpha\vert\leq n$ is precisely the condition
that the exponent vector $\alpha$ lies in $\Delta_n^m$.

\end{proof}

In general, we can relax the assumption in Theorem \ref{thm3-bijection} on $F$ being a field to $k$ being a UFD and having an equivalent structure that follows. To see why some restriction on the coefficient ring is necessary, consider $k = \mathbb{Z}_{6}$. Then $2 \cdot 3 = 0$ and in $k[x]$ we have $(2x)(3x)=0$. Thus factoring no longer controls absorbing behavior since zero divisors introduce further annihilation. 

\subsection{The Cayley Graph Interpretation} 

The exponent vector correspondence developed above identifies the set of principal monomial ideals with the lattice points of $\mathbb{N}^{m}$. This correspondence also preserves the local structure of the divisibility poset. In particular, the cover relations among principal monomial ideals are precisely the elementary lattice steps obtained by adding a standard basis vector. This observation allows us to identify the Hasse diagram with a Cayley graph. Under the correspondence $\alpha \leftrightarrow x^{\alpha}$, the edge represents multiplication by $x_{i}$.

\begin{theorem}\label{thm:1skeleton}
Let $F$ be a field and let $R=F[x_1,\dots,x_m]$. The Hasse diagram of the poset of principal monomial ideals ordered by reverse inclusion is canonically isomorphic to the $1$-skeleton of the Cayley graph
\[
\operatorname{Cay}(\mathbb{N}^{m},\{e_1,\dots,e_m\}).
\]

\end{theorem}

\begin{proof}
Define
\[
\Phi(\alpha)=(x^\alpha).
\]
By the previous subsection, $\Phi$ is a bijection between $\mathbb N^m$ and the principal monomial ideals and preserves the coordinatewise order.

Suppose
\[
(x^\beta)\subsetneq(x^\alpha).
\]
Then $\alpha\leq\beta$. This inclusion is a cover exactly when there is no intermediate exponent vector between $\alpha$ and $\beta$.

If $|\beta-\alpha|\geq2$, then there exists a nonzero proper vector $\delta$ with
\[
0<\delta<\beta-\alpha,
\]
giving an intermediate lattice point
\[
\alpha<\alpha+\delta<\beta.
\]
Hence the inclusion is not a cover.

Therefore a cover must satisfy
\[
|\beta-\alpha|=1,
\]
which implies
\[
\beta-\alpha=e_i
\]
for some standard basis vector.

Conversely, if
\[
\beta=\alpha+e_i,
\]
then no lattice point lies strictly between $\alpha$ and $\beta$, so the corresponding inclusion is a cover.

Thus the edges in the Hasse diagram are exactly
\[
\alpha\leftrightarrow\alpha+e_i.
\]
These are precisely the edges of
\[
\operatorname{Cay}(\mathbb N^m,\{e_1,\dots,e_m\}),
\]
so the Hasse diagram is canonically isomorphic to the $1$-skeleton of the Cayley graph.
\end{proof}

Combining this graph-theoretic interpretation with Theorem \ref{thm3-bijection}, the graded $n$-absorbing principal monomial ideals are precisely the vertices contained in the discrete simplex $\Delta_{n}^{m}$. Hence their divisibility structure is captured by the induced subgraph of the Cayley graph restricted to this simplex.

\begin{figure}[ht]
    \centering
    \includegraphics[width=\textwidth]{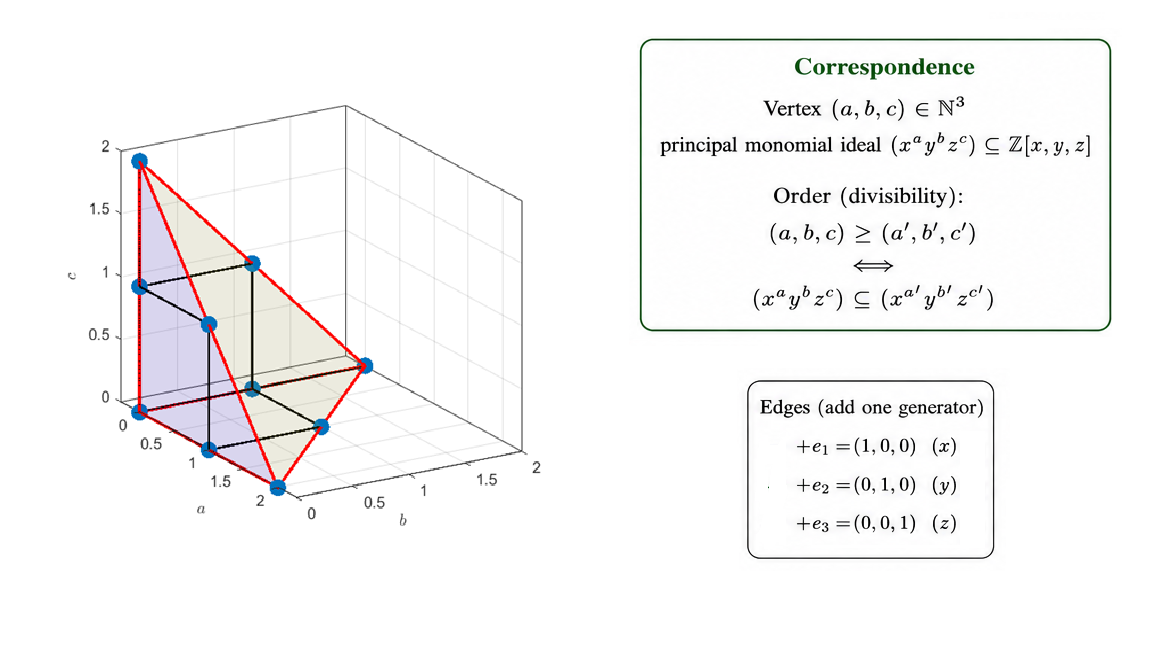}
    \caption{The Hasse diagram of principal monomial ideals in $\mathbb{Z}[x,y,z]$ is isomorphic to the $1$-skeleton of the Cayley graph $\operatorname{Cay}(\mathbb{N}^{3},\{e_{1},e_{2},e_{3}\})$. The highlighted vertices form the induced subgraph on the simplex $\Delta_{2}^{3}=\{(a,b,c)\in\mathbb{N}^{3}:a+b+c\leq2\}$, corresponding to the graded $2$-absorbing principal monomial ideals. The shaded faces are included only to visualize the surrounding tetrahedral region and do not represent edges of the $1$-skeleton.}
    \label{fig:1skeleton}
\end{figure}

\subsection{Consequences of the Model}

The simplex characterization provides more than a classification of graded $n$-absorbing principal monomial ideals; it gives a combinatorial parameter space in which several algebraic properties can be interpreted geometrically. In particular, the exponent vector $\alpha$ simultaneously records the absorbing behavior of the ideal, its position within the monomial lattice, and algebraic information arising from the generator $x^\alpha$.

The consequences developed below demonstrate the strength of this correspondence. First, combining the simplex characterization with known results on the relationship between $n$-absorbing and strongly $n$-absorbing ideals yields an equivalence between these notions in the principal monomial setting. Next, we examine the graded maximal ideal and show that its powers arise naturally as the collection of all monomial ideals lying on a fixed degree layer of the simplex. This perspective allows us to enumerate graded $n$-absorbing principal monomial ideals and connect the resulting counting functions to Hilbert series. Finally, we show that the same exponent-vector framework captures homological and dimension-theoretic information, linking the combinatorial position of a lattice point to invariants of the corresponding quotient ring.

\subsubsection{Equivalence with Graded Strongly n-Absorbing Ideals}

Within the class of principal monomial ideals, the distinction between graded and $n$-absorbance vanishes. The simplex characterization of Theorem \ref{thm3-bijection} shows that the graded $n$-absorbing condition is determined entirely by the exponent vector of the monomial generator $(x^{\alpha})$ is graded $n$-absorbing if, and only if, $\vert \alpha \vert \leq n$. Moreover, results of Choi and Walker \cite{choi2019n} and Secord \cite{secord24} imply that the same numerical condition governs $n$-absorbing and strongly $n$-absorbing behavior. Consequently, the discrete simplex
\[
\Delta_n^m=\{\alpha\in\mathbb{N}^m:|\alpha|\leq n\}
\]
serves as a common parameter space for all three notions of absorbance in the principal monomial setting. We summarize this equivalence below.

\begin{theorem}\label{thm:princgradedstrngabs}
 
  Let $F$ be a field, let $R = F[x_{1},x_{2},\dots,x_{m}]$ and let $I=(x^\alpha)$ be a proper principal monomial ideal of $R$. Then the following are equivalent:
    \begin{enumerate}
        \item $I$ is $n$-absorbing.
        \item $I$ is graded $n$-absorbing.
        \item $I$ is graded strongly $n$-absorbing.
        \item $\vert \alpha \vert \leq n$
        \item The vector $\alpha$ lies in the discrete simplex
        \[
        \Delta_{n}^{m} = \{\alpha \in \mathbb{N}^{m}: \vert \alpha \vert \leq n\}.
        \]
    \end{enumerate}
\end{theorem}

This equivalence is a consequence of the rigid combinatorial structure of principal monomial ideals and does not hold for arbitrary graded ideals. In general, graded and $n$-absorbance are distinct notions because the graded condition restricts the absorbing test to homogeneous elements.

\begin{proof}
The equivalence between (4) and (5) is immediate from the definition of
\(\Delta_n^m\). By Theorem \ref{thm3-bijection}, conditions (2), (4), and
(5) are equivalent.

By a result of Choi and Walker \cite{choi2019n}, a principal monomial ideal
\(I=(x^\alpha)\) satisfies
\[
\omega(I)=|\alpha|.
\]
Hence \(I\) is \(n\)-absorbing if and only if
\[
|\alpha|\leq n,
\]
which shows that (1), (4), and (5) are equivalent.

Finally, Secord's theorem \cite{secord24} establishes that every
\(n\)-absorbing ideal is strongly \(n\)-absorbing. Therefore, in this
setting,
\[
(1)\implies(3).
\]
Since strongly \(n\)-absorbing ideals are \(n\)-absorbing by definition, we
also have
\[
(3)\implies(1).
\]
Thus (1) and (3) are equivalent, completing the proof.
\end{proof}

\subsubsection{Powers of the Graded Maximal Ideal as Simplex Layers}

The simplex model also identifies the powers of the graded maximal ideal as the natural aggregation of principal monomial ideals of fixed absorbing index. We show that for graded maximal ideal $\mathcal{M}$ of $R$, taking powers of $\mathcal{M}$, denoted $\mathcal{M}^k$, directly corresponds to determining the graded absorption property of the ideal, which further identifies $\mathcal{M}^k$ as the sum of every principal monomial ideal of exact absorbing index $k$. 

\noindent Define
\[\omega_{gr}(I) = \textnormal{min}\{n \in \mathbb{N} : I \textnormal{ is graded $n$-absorbing}\}\] where we let $\omega_{gr}(I) = \infty$ if no such $n$ exists. We note that in general, $\omega_{gr}(I) \leq \omega(I)$ since the graded condition relies on only homogeneous elements. 
\begin{theorem}
   Let $F$ be a field, $R=F[x_1, \dots, x_m]$, and $\mathcal{M} = (x_1, \dots, x_m)$ be the graded maximal ideal of $R$. Then for every $k \geq 1:$
   \begin{enumerate}
       \item $\mathcal{M}^k$ is graded $k$-absorbing. For $k \geq 2$, $\mathcal{M}^k$ is not graded $(k-1)$-absorbing. For $k=1$, $\mathcal{M} $ is a graded prime ideal (and thus graded $1$-absorbing with $\omega_{gr}(\mathcal{M})=1$). 
       \item $\omega_{gr}{(\mathcal{M}^k)} = k$
      
   \end{enumerate}
\end{theorem}

\begin{proof}
    The ideal $\mathcal{M}^k$ is generated by all monomials of total degree equal to $k$. To see that $\mathcal{M}^k$ is graded $k$-absorbing, let $f_1, \dots, f_{k+1} \in h(R)$ be the nonzero homogeneous elements with $f_1\cdots f_{k+1} \in \mathcal{M}^k$. Each $f_i$ is a nonzero homogeneous polynomial, so either there exists some $i$ for which deg$(f_i)=0$ or deg$(f_i) \geq 1$ for all $i$. 

     If there exists $ f_i$ such that deg$(f_i)=0$, then $f_i \in F^{\times}$ is a unit. Thus the $k$-fold subproduct $\prod_{j\neq i} f_j = (f_1 \dots f_{k+1})/f_i \in \mathcal{M}^k$ and $\mathcal{M}^k$ is graded $k$-absorbing. 

   If for all $i$, deg$(f_i)\geq 1$, then each $f_i \in \mathcal{M}$. For fixed index $i$, the $k$-fold subproduct $\prod_{j\neq i}f_j$ has degree greater than or equal to $k$ (since $k$ terms each contribute degree $\geq 1)$, so $\prod_{j\neq i}f_j \in \mathcal{M}^k$ and we have that $\mathcal{M}^k$ is graded $k$-absorbing. 

    To see $\omega_{gr}{(\mathcal{M}^k)} = k$, first it is clear that when $k=1$, $\mathcal{M}^k$ is graded $1$-absorbing. %\mathcal{M}^k is maximal, and R/M is a field so \mathcal{M} is prime and thus graded 1-absorbing. 
    For $k\geq 2$, consider the $k$ homogeneous elements $f_1 = f_2 = \dots = f_k = x_1 \in h(R)$. Then deg$(f_1 \dots f_k) = k$, so $x_1^k \in \mathcal{M}^k$. If we take any $(k-1)$-fold subproduct of $f_i$'s, that will equal $x_1^{k-1}$, which has degree $k-1$, and thus $x_1^{k-1} \notin \mathcal{M}^k$. Thus $\mathcal{M}^k$ is not graded $(k-1)$-absorbing, so $\omega_{gr}(\mathcal{M}^k) \geq k$. Combined with (1), we see that $\omega_{gr}(\mathcal{M}^k)=k$. 
\end{proof}

These properties of $\mathcal{M}^k$ effect a bridge between the single ideal $\mathcal{M}^k$ and the full population of principal monomial ideals across all absorbing indices, which we show can be counted by a natural connection to the Hilbert series.

\subsubsection{Enumeration and Hilbert Series}

\begin{proposition}
    The generating function counting proper graded $k$-absorbing principal monomial ideals in $F[x_1, \dots, x_m]$ by the parameter $k$ is
    $$\sum_{k=1}^{\infty}\left(\binom{k+m}{m}-1\right)t^k = \frac{1}{(1-t)^{m+1}}-\frac{1}{1-t}$$ where $\binom{k+m}{m}-1$ counts the proper ideals. In particular, the count including the unit ideal satisfies $$\sum_{k=0}^{\infty}\binom{k+m}{m}t^k = \frac{1}{(1-t)^{m+1}}$$ which is the Hilbert series in $m+1$ variables. 
\end{proposition}
\begin{proof}
    By Theorem \ref{thm3-bijection} the proper graded $k$-absorbing principal monomial ideals correspond bijectively to the nonzero lattice of points in the simplex $\{\alpha \in \mathbb{N}^m : |\alpha| \leq k\}$. It is clear then that the total count including $\alpha=0$ is $|\{\alpha \in \mathbb{N}^m : |\alpha| \leq k\}| = \binom{k+m}{m}$. It is a straightforward calculation to see that subtracting the unit ideal contribution gives $\frac{1}{(1-t)^{m+1}}-\frac{1}{1-t}$.   
\end{proof}
From this proposition, we can deduce a count for the number of proper graded $n$-absorbing principal monomial ideals in $R$. 
\begin{corollary}
    Let $F$ be a field, and $R=F[x_1, \dots, x_m]$ Considering only the proper ideals (that is, $\alpha \neq 0$) we have the following: 
    \begin{enumerate}
        \item The number of proper graded $n$-absorbing principal monomial ideals in $R$ is $$\binom{n+m}{m}-1$$
        \item The number of principal monomial ideals with $\omega_{gr} = n$ is $$\binom{n+m-1}{m-1}$$
        \item The generating function over $n \geq 0$ and $m \geq 1$ counting the unit ideal (that is, including $\alpha = 0$) is $$\frac{s}{(1-t)(1-t-s)}$$
    \end{enumerate}
\end{corollary}

This counting argument shows that $\mathcal{M}^k$ can be identified as the largest ideal built from principal pieces that are minimally $k$-absorbing, sitting above all $\binom{k+m-1}{m-1}$ of them in the divisibility poset of Theorem \ref{thm:nmlatticehasse}; $\mathcal{M}^k$ is the sum of all principal ideals $(x^{\alpha})$ with $|\alpha|=k$, thus $\mathcal{M}^k$ is the smallest ideal that contains all such ideals simultaneously. 

\begin{corollary}
    For every $k \geq 1$, $\mathcal{M}^k = \sum\limits_{|\alpha|=k} (x^{\alpha})$, a sum of exactly $\binom{k+m-1}{m-1}$ principal monomial ideals, each of exact graded absorbing index $k$. 
\end{corollary}

\textbf{Example 1:} Consider $F[x]$, $(m=1)$. Then the number of proper graded $n$-absorbing principal monomial ideals is $n$: $(x), (x^2), \dots, (x^n)$, and the number of principal monomial ideals where $\omega_{gr} = n$ is just 1, namely, $(x^n)$. 

\textbf{Example 2:} Consider $F[x,y], (m=2)$. Then the number of proper graded $n$-absorbing principal monomial ideals is $\binom{n+2}{2}-1 = \frac{(n+1)(n+2)}{2}-1$. For $n=2$, there are five such ideals: $(x), (y), (x^2), (xy), (y^2)$ and these correspond to the five nonzero lattice points in $\{a+b \leq 2, (a,b) \neq (0,0)\} \subset \mathbb{N}^2$. 

\textbf{Example 3:} Consider $F[x,y,z], (m=3)$. For $n=2$, we see that there are $9$ proper, graded $2$-absorbing principal monomial ideals: $(x)$, $(y)$, $(z)$, $(x^2)$, $(y^2)$, $(z^2)$, $(xy)$, $(xz)$, $(yz)$. There are exactly three ideals where $\omega_{gr} = 1$: $(x), (y), (z)$ (which matches the count $\binom{3}{2}=3$). There are six ideals where $\omega_{gr}=2$: $(x^2)$, $(y^2)$, $(z^2)$, $(xy)$, $(xz)$, $(yz)$ (which matches the count $\binom{4}{2} = 6$).

\subsubsection{Algebraic Invariants Associated to the Exponent Vector}
The simplex characterization identifies graded $n$-absorbing principal monomial ideals through the numerical condition $\vert \alpha \vert$. Since the same exponent vector determines the degree of the monomial generator, it is natural to ask whether this combinatorial description also records algebraic invariants of the corresponding quotient. For principal monomial ideals, the answer is particularly simple: the total degree of $\alpha$ determines the unique nontrivial shift in the minimal graded free resolution of $R/I$

The inclusion of homological invariants is motivated by the fact that the simplex model does not merely classify absorbing behavior; it identifies exponent vectors as a common parameter space for both ideal-theoretic and homological properties of principal monomial ideals.

\begin{proposition} \label{prop:homologicalinv}
Let \( I=(x_1^{a_1}\cdots x_m^{a_m}) \subseteq R =F[x_1,\dots,x_m] \) and let $\alpha=(a_1,\dots,a_m).$  Then \[ \dim(R/I) = m-1 \] and the minimal graded free resolution of \(R/I\) is \[ 0 \to R(-\vert \alpha \vert ) \to R \to R/I \to 0. \] Consequently, \[ \beta_{0,0}(R/I)=1, \qquad \beta_{1,|\alpha|}(R/I)=1, \] and all other graded Betti numbers vanish. 
\end{proposition}

\begin{proof}
Let $I = (x_{1}^{a_{1}} x_{2}^{a_{2}}\cdots x_{m}^{a_{m}}) = (x^{\alpha}) \subseteq R = F[x_{1},x_{2},\dots,x_{m}]$, where $\alpha = (a_{1},a_{2},\dots,a_{m}$). Since $I$ is a nonzero proper principal ideal of the polynomial ring, Krull's Principal Ideal Theorem implies that every minimal prime over $I$ has height 1 \cite{atiyah2018introduction}. Since $\dim(R)=m$, it follows that \cite{atiyah2018introduction}

\begin{equation*}
    \dim(R/I)=\dim(R)-\operatorname{ht}(I) = m-1. 
\end{equation*} 

Now let $d = \vert \alpha \vert = a_{1}+a_{2}+\dots+a_{m}$. Since $I=(x^{\alpha})$ is generated by the homogeneous monomial $x^{\alpha}$ of degree $d$, multiplication by $x^{\alpha}$ defines an injective graded homomorphism $\varphi$ where

\begin{equation*}
    \varphi:R(-d)\xrightarrow{\cdot x^\alpha}R.
\end{equation*}

Because $R$ is an integral domain, $\ker(\varphi) = 0$. Moreover, $\operatorname{Im}(\varphi) = x^{\alpha}R = (x^{\alpha}) = I$. Thus we obtain the exact sequence  

\begin{equation} 
0
\longrightarrow
R(-d)
\xrightarrow{\varphi}
R
\xrightarrow{\pi}
R/I
\longrightarrow
0,
\end{equation}
where $\pi$ denotes the natural quotient map \cite{eisenbud2013commutative, peeva2010graded}. Since $R(-d)$ and $R$ are free $R$-modules, this exact sequence is a graded free resolution of $R/I$ \cite{peeva2010graded}. Moreover, the matrix of the differential is the $1 \times 1$ matrix $[x^{\alpha}]$. Since $\deg(x^{\alpha}) = d > 0$, we have $x^{\alpha} \in (x_{1},\dots,x_{m})$. Thus every entry of every differential belongs to the homogeneous maximal ideal \cite{peeva2010graded}. By the standard criterion for minimal graded free resolutions over $F[x_{1},\dots,x_{m}]$, the resolution is minimal \cite{eisenbud2013commutative,peeva2010graded}. Therefore the minimal graded free resolution of $R/I$ is

\begin{equation*}
    0
\longrightarrow
R(-|\alpha|)
\longrightarrow
R
\longrightarrow
R/I
\longrightarrow
0.
\end{equation*}

It follows that the only nonzero graded Betti numbers are

\begin{equation*}
    \beta_{0,0}(R/I)=1
\qquad\text{and}\qquad
\beta_{1,|\alpha|}(R/I)=1,
\end{equation*}
while
\begin{equation*}
    \beta_{i,j}(R/I)=0
\end{equation*}
for all other pairs $(i,j)$ \cite{peeva2010graded}.
\end{proof}

The connection to the theory of graded $n$-absorbing ideals arises through Theorem \ref{thm3-bijection}, which identifies graded $n$-absorbing principal monomial ideals with the lattice points of the simplex
\[
\Delta_n^m=\{\alpha\in\mathbb N^m:|\alpha|\leq n\}.
\]
For example, consider the principal monomial ideal
\[
I=(x^{2}y)\subseteq F[x,y].
\]
The corresponding exponent vector is $\alpha=(2,1)$, so
\[
\vert \alpha \vert =3.
\]
Thus $\alpha\in\Delta_{3}^{2}$, and by Theorem \ref{thm3-bijection}, $I$ is a graded $3$-absorbing ideal. At the same time, Proposition \ref{prop:homologicalinv} shows that the total degree $\vert\alpha\vert=3$ determines the nontrivial shift in the minimal graded free resolution:
\[
0\longrightarrow R(-3)\longrightarrow R\longrightarrow R/I
\longrightarrow 0,
\]
with
\[
\beta_{0,0}(R/I)=1
\qquad\text{and}\qquad
\beta_{1,3}(R/I)=1.
\]
Thus, in this example, the same exponent vector that determines the $n$-absorbing behavior through the simplex condition also determines the nontrivial graded shift and Betti number of the quotient.

\section{Examples and Geometric Interpretation}

The simplex characterization allows graded $n$-absorbing principal monomial ideals to be studied through the geometry of their exponent vectors. Beyond providing a criterion for determining whether a principal monomial ideal is graded $n$-absorbing, this viewpoint reveals how the position of a lattice point encodes additional algebraic information. We conclude this section with examples illustrating the hierarchy of graded $n$-absorbance, the sharpness of the simplex boundary, and the recovery of homological invariants from the associated lattice point.

\

%3 absorbing which is not 2 absorbing as per JPAA request

\begin{example}[A graded 3-absorbing ideal which is not graded 2-absorbing]

Consider the principal monomial ideal $I = (x^{2}y) \subseteq F[x,y,z]$. The corresponding exponent vector is $(2,1,0) \in \mathbb{N}^{3}$. Since $\vert \alpha \vert = 2+1 = 3$, by Theorem \ref{thm3-bijection} it follows that $I$ is a graded 3-absorbing ideal. Furthermore, since $\vert \alpha \vert = 3 > 2$, the same characterization shows that $I$ cannot be a graded $2$-absorbing ideal. Hence, $I=(x^{2}y)$ provides an explicit example of a graded $3$-absorbing ideal which is not graded $2$-absorbing. In particular, this demonstrates that the framework developed here extends naturally beyond the graded $2$-absorbing case and captures genuinely higher-order absorbance phenomena. 

\end{example}

%Geometric Interpretation

\begin{example}[Crossing the simplex boundary]
This example illustrates the sharp geometric boundary between graded 3-absorbing and graded 4-absorbing principal monomial ideals. Consider the collection $\mathcal{S} \subseteq \mathbb{N}^{3}$ of exponent vectors given by
\[
\begin{aligned}
\mathcal{S}=\{&(3,0,0),(0,3,0),(0,0,3),(2,1,0),(1,2,0),\\
&(2,0,1),(1,0,2),(0,2,1),(0,1,2),(1,1,1)\}.
\end{aligned}
\]
These are precisely the lattice points $\alpha = (a,b,c) \in \mathbb{N}^{3}$ satisfying $ \vert \alpha \vert = a+b+c=3$, and therefore correspond to principal monomial ideals of the form $(x^{a}y^{b}z^{c})$ that are graded 3-absorbing by Theorem \ref{thm3-bijection}. Geometrically, these vectors are precisely the lattice points lying on the boundary layer $a+b+c=3$ of the simplex $\Delta_{3}^{3} = \{\alpha \in \mathbb{N}^{3} : \vert \alpha \vert \leq 3 \}$ as illustrated in Figure \ref{fig:hasse_simplex}.

\begin{figure}[ht]
\centering

\begin{subfigure}[t]{0.49\textwidth}
    \centering
    \includegraphics[width=\linewidth]{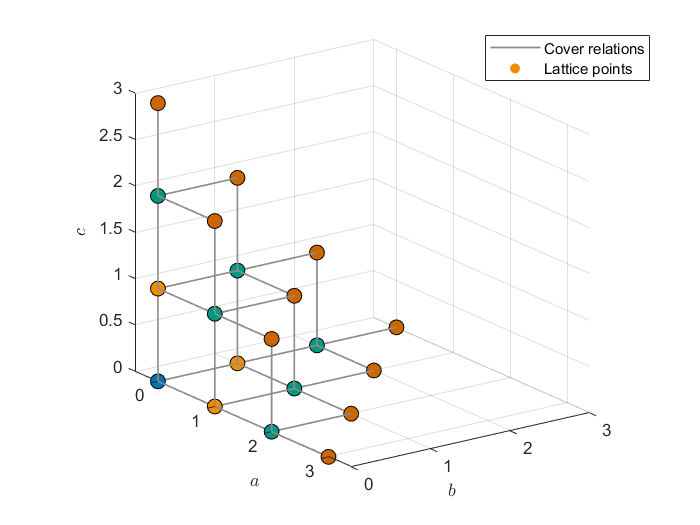}
    \caption{Hasse diagram of the truncated monomial lattice.}
\end{subfigure}
\hfill
\begin{subfigure}[t]{0.49\textwidth}
    \centering
    \includegraphics[width=\linewidth]{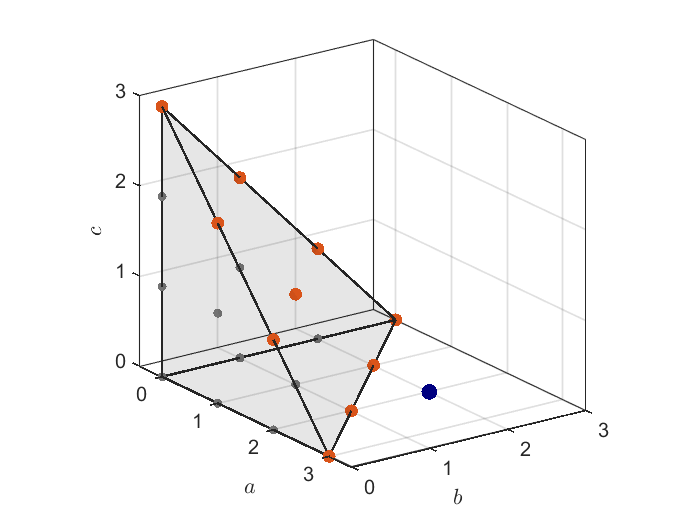}
    \caption{Lattice simplex representation.}
\end{subfigure}

\caption{Two complementary representations of principal monomial ideals in $F[x,y,z]$. Panel (A) depicts the cover relations in the truncated monomial lattice, while Panel (B) shows the corresponding lattice points in the simplex $\Delta_{3}^{3} = \{\alpha\in\mathbb{N}^3:|\alpha|\le 3\}$.}
\label{fig:hasse_simplex}
\end{figure}

Now consider the exponent vector $\beta = (2,2,0)$ corresponding to the principal monomial ideal $(x^{2}y^{2}) \subseteq F[x,y,z]$. Since $\vert \beta \vert = 2+2+0=4 > 3$, by Theorem \ref{thm3-bijection}, it is a graded 4-absorbing but not graded 3-absorbing ideal. Consequently the point $(2,2,0) \notin \Delta_{3}^{3}$ as illustrated in Figure \ref{fig:hasse_simplex} (B). This illustrates the sharp geometric boundary established by Theorem \ref{thm3-bijection};  principal monomial ideals are graded $n$-absorbing precisely when their exponent vectors lie in the discrete simplex $\Delta_{n}^{m}$. 

\end{example}

.
\begin{example}[Recovering algebraic and homological data through the simplex model]

Consider the principal monomial ideal $I=(x^{2}yz) \subseteq F[x,y,z]$. The corresponding exponent vector $\alpha = (2,1,1) \in \mathbb{N}^{3}$. By Theorem \ref{thm3-bijection}, the graded $n$-absorbing behavior of $I$ is determined entirely by the total degree of this exponent vector. Since $\vert \alpha \vert = 2+1+1 =4$, the ideal $I$ is graded 4-absorbing. Moreover, since $\vert \alpha \vert = 4 > 3$, we have that $\alpha \notin \Delta_{3}^{3}$ and therefore $I$ is not graded 3-absorbing. 

Geometrically, the exponent vector $\alpha = (2,1,1)$ corresponds to a lattice point in $\mathbb{N}^{3}$ lying on the boundary layer $\vert \alpha \vert =4$ of the discrete simplex 
\[
\Delta_{4}^{3} = \{\alpha \in \mathbb{N}^{3}: \vert \alpha \vert \leq 4\}.
\]
Thus, the simplex model recovers the exact absorbance level of the ideal. That is, the smallest simplex containing the exponent vector $\alpha = (2,1,1)$ is precisely $\Delta_{4}^{3}$ and hence $4 = \min \{k: \alpha \in \Delta_{k}^{3}\}$. This illustrates the nested hierarchy established in Theorem \ref{thm:mn}; increasing the absorbing parameter enlarges the corresponding simplex and incorporates additional expont vectors.

The exponent vector also provides information beyond the graded $n$-absorbing classification. In particular, the same combinatorial data that determines the location of the ideal within the nested hierarchy also determines algebraic invariants of the quotient ring. While graded $n$-absorbance is govered by the inequality $\vert \alpha \vert \leq n$, the total degree $\vert \alpha \vert$ determines the unique nontrivial degree shift appearing in the minimal graded free resolution of $R/I$. By Proposition \ref{prop:homologicalinv}, $\dim(F[x,y,z]/(x^{2}yz)) = 3-1 = 2$ and the minimal graded free resolution is 
\[
0\longrightarrow F[x,y,z](-4)
\xrightarrow{\cdot x^{2}yz} F[x,y,z]
\longrightarrow F[x,y,z]/(x^{2}yz)
\longrightarrow 0.
\]
Consequently, the only nonzero graded Betti numbers are 
\[
\beta_{0,0}=1,\qquad \beta_{1,4}=1.
\]

Therefore, the exponent vector $\alpha=(2,1,1)$ simultaneously records three types of information: its location in the simplex filtration determines the graded $n$-absorbing level, its coordinates determine the combinatorial position of the corresponding monomial in the lattice, and its total degree determines the graded shift in the minimal free resolution.

\end{example}

The examples above illustrate three complementary aspects of the simplex characterization as summarized in Table \ref{tab:simplex_examples} given below.

\begin{table}[ht]
\centering
\begin{tabular}{c|c|c|c|c}
\hline
\textbf{Example} &
\textbf{Ideal} &
\textbf{Exponent Vector} &
\textbf{Simplex Interpretation} &
\textbf{Homological Data} \\
\hline

1 &
$(x^{2}y)$ &
$(2,1,0)$ &
\makecell{
$|\alpha|=3$\\
$\alpha\in\Delta_3^3\setminus\Delta_2^3$\\
graded $3$-absorbing\\
not graded $2$-absorbing
}
&
\makecell{
$\dim(R/I)=2$\\
$\beta_{0,0}=1$\\
$\beta_{1,3}=1$
}
\\
\hline

2 &
$(x^{2}y^{2})$ &
$(2,2,0)$ &
\makecell{
$|\alpha|=4$\\
$\alpha\in\Delta_4^3\setminus\Delta_{3}^{3}$\\
graded $4$-absorbing
}
&
\makecell{
$\dim(R/I)=2$\\
$\beta_{0,0}=1$\\
$\beta_{1,4}=1$
}
\\
\hline

3 &
$(x^{2}yz)$ &
$(2,1,1)$ &
\makecell{
$|\alpha|=4$\\
$\alpha\in\Delta_4^3\setminus\Delta_3^3$\\
minimal simplex level $4$\\
graded $4$-absorbing
}
&
\makecell{
$\dim(R/I)=2$\\
$\beta_{0,0}=1$\\
$\beta_{1,4}=1$
}
\\

\hline
\end{tabular}
\caption{Summary of examples illustrating the simplex characterization and associated algebraic invariants of graded $n$-absorbing principal monomial ideals.}
\label{tab:simplex_examples}
\end{table}
\

These examples illustrate the central philosophy of this paper. Once principal monomial ideals are identified with lattice points, determining whether a principal monomial ideal is graded $n$-absorbing becomes a geometric problem: by Theorem \ref{thm3-bijection}, the ideal is graded $n$-absorbing precisely when its corresponding exponent vector lies in the discrete simplex $\Delta_{n}^{m}$. The same exponent vector also encodes additional algebraic information; its total degree determines both the minimal graded $n$-absorbing level and the graded shift in the minimal free resolution, while its coordinates record the combinatorial position of the monomial within the ambient lattice.

\section{Conclusion and Suggestions for Future Work}
The principal contribution of this paper is not merely the classification of graded $n$-absorbing principal monomial ideals, but rather the observation that graded $n$-absorbance admits a natural combinatorial realization in the principal monomial setting. Here, this realization takes the form of lattice points inside simplices together with their associated Hasse diagrams, providing a geometric interpretation of an algebraic property that had previously been studied primarily through ideal-theoretic methods. Consequently, the natural next stage of this research program is not simply to classify additional classes of ideals, but rather to identify the appropriate combinatorial models that govern graded $n$-absorbing behavior in increasingly general settings.

The first direction is to extend this framework beyond principal monomial ideals. Following the work of Choi \cite{choi2022n, choi2024n}, the next natural setting is that of finitely generated monomial ideals. Although this appears to be only a modest generalization, it represents a substantial increase in complexity. Throughout the present work, every principal monomial ideal is represented by a single exponent vector $\alpha\in\mathbb{N}^{m}$, allowing graded $n$-absorbance to be characterized by the simple condition $|\alpha|\leq n$. Consequently, the simplex naturally emerges as the parameter space for the theory. Once several generators are permitted, this description immediately breaks down. If
\[
I=(x^{\alpha_1},x^{\alpha_2},\dots,x^{\alpha_r}),
\]
where $\alpha_{i} \in \mathbb{N}^{m}$. The ideal can no longer be represented by a single lattice point but instead by a finite collection of exponent vectors whose interaction determines the algebraic structure. For example, the ideal
\[
(xy,xz,yz)
\]
is generated entirely by degree-two monomials, yet its ideal-theoretic properties are governed by the overlap of its generators rather than by their common degree. Thus there is no immediate analogue of the condition $|\alpha|\leq n$, suggesting that the simplex is no longer sufficient as a combinatorial parameter space.

A particularly promising intermediate step is the study of squarefree monomial ideals. By the Stanley-Reisner correspondence, squarefree monomial ideals are naturally associated with simplicial complexes \cite{francisco2014survey,garsia1984group}. Consequently, one may seek combinatorial characterizations of graded $n$-absorbing ideals in terms of distinguished structures within the associated simplicial complexes, such as faces, vertex covers, or other combinatorial invariants. Such a correspondence would provide a natural bridge between the simplex geometry developed in this paper and the richer combinatorics of arbitrary monomial ideals.

Beyond the squarefree setting, finitely generated monomial ideals suggest that the appropriate combinatorial model should no longer consist of individual lattice points but of geometric objects encoding collections of exponent vectors. Possible candidates include finite antichains, order ideals, hypergraphs, Newton polyhedra, and polyhedral cell complexes. Likewise, the Hasse diagram interpretation developed in this paper becomes substantially more subtle. While multiplication by a variable naturally corresponds to moving along a single edge for principal monomial ideals, no canonical elementary move exists once multiple generators are present. Different notions of adjacency lead to genuinely different graph structures, making the identification of the correct combinatorial model itself an interesting problem. Since many algebraic properties of monomial ideals are already encoded by Newton polyhedra and polyhedral cell complexes, these presently appear to be among the most promising candidates.

A second major direction is to move beyond monomial ideals altogether by considering principal binomial ideals, and ultimately arbitrary binomial ideals. This extension is fundamentally different in nature because binomial ideals encode not only combinatorial information through exponent differences but also algebraic information arising from relations among terms. For example,
\[
(x^{n}-y^{n})
\qquad\text{and}\qquad
(x^{n-1}y-y^{n})
\]
have the same degree but exhibit very different factorization behavior. Since graded $n$-absorbing ideals are defined through products of elements, relations among terms and cancellation phenomena become essential features, and the degree-counting arguments employed throughout the present work can no longer suffice.

Accordingly, a richer geometric framework is required. If $L\subseteq\mathbb{Z}^{m}$ is a lattice, the associated lattice ideal
\[
I_{L}=(x^{u}-x^{v}: u,v \in \mathbb{N}^{m}, \ u-v\in L)
\]
is governed by lattice relations in $\mathbb{Z}^{m}$ rather than exponent vectors in $\mathbb{N}^{m}$. Consequently, the natural geometry shifts from simplices to quotient lattices and toric structures, and the Hasse diagram interpretation developed here is replaced by geometry arising from lattice quotients and toric ideals. For example, over an algebraically closed field of characteristic not dividing $n$,
\[
(x^{n}-y^{n})
=
\bigcap_{k=0}^{n-1}(x-\zeta^{k}y),
\]
where $\zeta$ is an $n$th root of unity, illustrating that the underlying geometry now depends upon both lattice relations and the arithmetic of the coefficient field. Understanding graded $n$-absorbance in this setting therefore appears closely connected to the theory of lattice ideals, toric ideals, and Cayley graphs of finitely generated abelian groups.

Finally, beyond the classification of ideals themselves lies the broader objective of understanding how algebraic invariants are encoded by the resulting combinatorial geometry. For principal monomial ideals, invariants such as the Hilbert series, Betti numbers, and Krull dimension are readily computed and largely reflect the simplicity of principality. As the parameter spaces become more sophisticated, however, these invariants should become genuine geometric quantities rather than straightforward algebraic calculations. This suggests the broader conjectural principle that important classes of graded $n$-absorbing ideals may admit natural combinatorial parameter spaces in which algebraic properties are reflected geometrically. The principal monomial case provides evidence for this philosophy: the same exponent vector simultaneously encodes the graded $n$-absorbing level, the combinatorial position of the ideal, and basic homological invariants of the quotient. If established, such a framework would provide a unified geometric perspective on graded $n$-absorbing ideals extending well beyond the principal monomial setting developed in this paper.

\bibliographystyle{elsarticle-num}
\bibliography{references}

\end{document}